\documentclass{amsart}
\usepackage[utf8]{inputenc}

\usepackage[english]{babel}
\usepackage[utf8]{inputenc}
\usepackage{geometry}
\usepackage{amsmath,amsfonts,amssymb,amsthm} 
\usepackage{mathrsfs}  
\usepackage{braket} 
\usepackage{paralist}
\usepackage{graphicx}
\usepackage{placeins} 
\usepackage{booktabs}
\usepackage{tikz}
\usepackage{xspace}
\usepackage{mathtools}
\usepackage{todonotes}
\usepackage{enumerate}

\newtheorem{teo}{\textbf\textrm\textsc{Theorem}}[section]
\newtheorem{prop}[teo]{\textbf\textrm\textsc{Proposition}}
\newtheorem{lemma}[teo]{\textbf\textrm\textsc{Lemma}}

\theoremstyle{definition}

\newtheorem{oss}[teo]{\textbf\textrm\textsc{Remark}}

\newcounter{cst}
\newcommand{\ctel}[1]{c_{\refstepcounter{cst}\label{#1}\thecst}}
\newcommand{\cter}[1]{c_{\ref{#1}}}

\def\rw{\rightarrow}
\def\ru{\rightharpoonup}

\def\d{\mathcal{D}}

\def\e{\varepsilon}

\def\G{\Gamma}

\def\f{\varphi}

\def\un{\mathbf{1}}
\def\pa{\partial}
\def\p{\psi}
\def\P{\Psi}
\def\F{\Phi}
\def\n{\nabla}
\def\R{\mathbb{R}}

\def\ds{\displaystyle}
\def\be{\begin{equation}}
\def\ee{\end{equation}}
\def\ov#1{\overline{#1}}

\newcommand{\oo}{\infty}
\newcommand{\lt}{\left}
\newcommand{\rt}{\right}

\newcommand{\bk}{\color{black}}

\newcommand{\rd}{}

\date{}

\title[Mathematical analysis of a reduced corrosion model]{Mathematical analysis of a thermodynamically consistent reduced model for iron corrosion}
\author[C. Canc\`es]{Cl\'ement Canc\`es}
\address{Cl\'ement Canc\`es ({\tt clement.cances@inria.fr}): Univ. Lille, Inria, CNRS, UMR 8524 - Laboratoire Paul Painlevé, F-59000 Lille, France.}

\author[C. Chainais-Hillairet]{Claire Chainais-Hillairet}
\address{Claire Chainais-Hillairet ({\tt claire.chainais@univ-lille.fr}): Univ. Lille, CNRS, Inria, UMR 8524 - Laboratoire Paul Painlevé, F-59000 Lille, France.}

\author[B. Merlet]{Beno\^{\i}t Merlet}
\address{Beno\^{\i}t Merlet ({\tt benoit.merlet@univ-lille.fr}): Univ. Lille, CNRS, Inria, UMR 8524 - Laboratoire Paul Painlevé, F-59000 Lille, France.}

\author[F. Raimondi]{Federica Raimondi}
\address{Federica Raimondi ({\tt federica.raimondi@inria.fr}): Univ. Lille, CNRS, Inria, UMR 8524 - Laboratoire Paul Painlevé, F-59000 Lille, France.}

\author[J. Venel]{Juliette Venel}
\address{Juliette Venel ({\tt juliette.venel@uphf.fr}): Univ. Polytechnique Hauts-de-France, CERAMATHS, FR CNRS 2037, F-59313 Valenciennes, France.}

\begin{document}

\begin{abstract}
We are interested in a reduced model for corrosion of iron, in which ferric cations and electrons evolve in a fixed oxide layer subject to a self-consistent electrostatic potential. Reactions at the boundaries are modeled thanks to Butler-Volmer formulas, whereas the boundary conditions on the electrostatic potential model capacitors located at the interfaces between the materials. Our model takes inspiration in existing papers, to which we bring slight modifications in order to make it consistent with thermodynamics and its second principle. Building on a free energy estimate, we establish the global in time existence of a solution to the problem without any restriction on the physical parameters, in opposition to previous works. The proof further relies on uniform estimates on the chemical potentials that are obtained thanks to Moser iterations. Numerical illustrations are finally provided to highlight the similarities and the differences between our new model and the one previously studied in the literature. 
\end{abstract}

\maketitle
\noindent
\\[5pt]

\section{Introduction}\label{sec.intro}

\subsection{General framework of the study}

At the request of the French nuclear waste management agency ANDRA, investigations are carried out in order to evaluate the long-term safety of the geological repository of radioactive waste. The context is the following: the waste is confined in a glass matrix, placed into steel canisters and then stored in a claystone layer at a depth of several hundred of meters. The long-term safety assessment of the geological repository has to take into account the degradation of the carbon steel
used for the waste overpacks and the cell disposal liners, which are in contact with the claystone formation. The study of the corrosion processes that arise at the surface of the steel canisters  and of the cell disposal liners takes part in the modelling and simulation of the repository. This has motivated the introduction of the Diffusion Poisson Coupled Model (DPCM) by Bataillon {\em et al.} in~\cite{Eacta}.

The DPCM describes the oxidation of 
a metal covered by an oxide  layer (magnetite) in contact with the claystone. It consists in a system of drift-diffusion equations on the density of charge carriers (electrons, $F\!e^{3+}$ cations and oxygen vacancies) coupled with a Poisson equation on the electric potential. Boundary conditions of Robin-Fourier type are prescribed by the electrochemical reactions and the potential drops at the interfaces with the claystone and with the metal.
They induce additional couplings in the system and involve numerous physical parameters. The system also includes equations governing the motion of the boundaries.

Up to now, no existence result has been established for the general DPCM~\cite{Eacta}. However, a few mathematical results on some simplified versions  have been obtained. The well-posedness of a two-species (electrons and ferric cations) model set on a fixed domain has been established in~\cite{CHLV-AML} for the stationary case and in~\cite{CHLV-DCDS} for the evolutive case. Some numerical methods for the simulation of DPCM have been introduced in~\cite{JCP} and implemented in the code CALIPSO. The numerical experiments proposed in~\cite{Eacta, JCP} have highlighted the long-time behaviour of the model: after a transient period, the size of the oxide layer stays constant, while both interfaces move at the same velocity and the densities of charge carriers as the electrical potential have stationary profiles. This corresponds to a traveling wave solution to DPCM. The existence of a traveling wave solution has been proved first on a simplified version of DPCM in~\cite{CHG-NARWA}. Recently in~\cite{BCHZ}, the existence is also established for the original DPCM thanks to a computer-assisted proof.

The description of the transport of charge carriers in the oxide layer proposed by the DPCM is similar to the transport of charge carriers in a semiconductor device as proposed by van Roosbroeck~\cite{vR1950}. The differences come first from some reaction terms due to recombination-generation in the continuity equations and mainly from the boundary conditions, which are of Dirichlet-Neumann type in the semiconductor setting. The mathematical analysis of the drift-diffusion-Poisson system of equations for semiconductor devices is the subject of a series of seminal papers by Gajewski and Gr\"{o}ger~\cite{GG_1986,GG_1989, GG_1990}. The strategy of proof proposed by Gajewski and Gr\"{o}ger relies strongly on the underlying variational structure of the model, in agreement with thermodynamics. One keypoint is a convex functional which can be interpreted as a free energy from the viewpoint of thermodynamics . This strategy has been further used in many papers dealing for instance  with electro-reaction-diffusion  processes~\cite{GG_1996, GGH_1996,GH_1997}, spin-polarized drift-diffusion models~\cite{Glitzky_2008}, electronic models for solar cells~\cite{Glitzky_2011,Glitzky_2012},...               

In order to obtain some mathematical results for the DPCM with a similar strategy, it is crucial to understand the impact of the moving boundary equations and of the boundary conditions on the structure of the system. Actually, the derivation of the DPCM proposed in~\cite{Eacta} does not rely on energetic considerations. As a consequence, its thermodynamic stability is unclear and neither a satisfactory well-posedness result nor the assessment of the long-time behavior of the system have been established so far. 
\rd 
This paper was thought as a first step towards the derivation and the analysis of a thermodynamically consistent DPCM model. Here, we restrict our attention to a simplified setting with only two species (electrons and iron cations), the goal being to get consistent couplings 
between the boundary conditions and the bulk equations. Since oxygen vacancies are not considered in our simplified setting, the boundaries are fixed throughout this paper. The derivation of a thermodynamically consistent counterpart to the full DPCM model will be the purpose of a future work. 
\bk

\subsection{From the original two-species DPCM to the new model}\label{sec.DPCMvDPCM}

Let us start by introducing the original two-species DPCM, which is set on the  fixed domain $(0,1)\subset\R$. It is already a challenge on this simplified model to manage the complexity of the boundary conditions in order to establish the dissipative behaviour of the associate 
system of partial differential equations.

We will denote by $u_1$ the density of ferric cations, $u_2$ the density of electrons and $v_0$ the electric potential. 
We consider a scaled model, which involves only dimensionless constants and scaled parameters,  detailed below.
The original DPCM in this case can be written, for $t\ge 0$,
\begin{subequations}\label{DPCM}
\begin{flalign}
&&&&\pa_t u_1 +\pa_x J_1&=0\quad\text{ with }\quad J_1=-d_1(\pa_x u_1 +z_1 u_1 \pa_x v_0)&\mbox{in } (0,1),&&\label{DPCM.u1}\\
&&&&\pa_t u_2 +\pa_x J_2&=0\quad\text{ with }\quad J_2=-d_2(\pa_x u_2 + z_2 u_2 \pa_x v_0)&\mbox{in } (0,1),&&\label{DPCM.u2}\\
&&&&-\lambda^2 \pa^2_{xx} v_0&= z_1 u_1 + z_2 u_2 +\rho_\textrm{hl} &\mbox{in }  (0,1),&&
\end{flalign}
\end{subequations}
with the following boundary conditions:
\begin{subequations}\label{CLDPCM}
\begin{align}
-&J_1(0)=k_1^0u_1(0) e^{\displaystyle\frac{z_1}{2}v_0(0)}-m_1^0({\ov u}_1-u_1(0))e^{-\displaystyle\frac{z_1}{2}v_0(0)},\label{CLDPCM.u1.0}\\
&J_1(1)=m_1^1 u_1(1)e^{\displaystyle\frac{z_1}{2}(v_0(1)-V)}-k_1^1({\ov u}_1-u_1(1))e^{-\displaystyle\frac{z_1}{2}(v_0(1)-V)},\label{CLDPCM.u1.1}\\
-&J_2(0)=k_2^0 u_2(0) e^{\displaystyle\frac{z_2}{2}v_0(0)}-m_2^0 e^{-\displaystyle\frac{z_2}{2}v_0(0)},\label{CLDPCM.u2.0}\\
&J_2(1)=m_2^1 u_2(1)-k_2^1 {\ov u}_2^\textrm{met} \log \left(1+e^{\displaystyle z_2(V-v_0(1))}\right),\label{CLDPCM.u2.1}\\
&v_0(0)-\alpha_0\pa_x v_0(0)=\Delta\P_0^\textrm{pzc},\label{CLDPCM.v0.0}\\
&v_0(1)+\alpha_1 \pa_x v_0(1)= V-\Delta \P_1^\textrm{pzc}.\label{CLDPCM.v0.1}
\end{align}
\end{subequations}
The system~\eqref{DPCM}-\eqref{CLDPCM} is supplemented with initial conditions on $u_1$ and $u_2$:
\be\label{CI}
u_1(t=0)= u_1^\textrm{in},\quad u_2(t=0)=u_2^\textrm{in}.
\ee

Let us comment the different parameters arising in the system:
\begin{itemize}
\item $\lambda^2$ is the scaled Debye length, $\alpha_0$ and $\alpha_1$ are positive dimensionless parameters related to the differential capacitances of the interfaces;
\item $\rho_\textrm{hl}$ is the net charge density of the ionic species in the host lattice, assumed to be constant in space, $\rho_\textrm{hl}=-5$;
\rd
\item $z_i$ is the (dimensionless) charge of the $i^\textrm{th}$ species. In our setting, $z_1 = +3$ and $z_2 = -1$; 
\bk 
\item $d_1$ and $d_2$ are the scaled diffusion coefficients. In practice the scaling is relative to the characteristic time of cations, $d_1=1$ and $d_2\gg 1$;
\item ${\ov u}_1$ is the maximum occupancy for octahedral cations in the host lattice;
\item ${\ov u}_2^{\textrm{met}}$ is the electron density of state of the metal (Friedel model);
\item $(k_i^{\Gamma},m_i^{\Gamma})_{i={1,2}, \Gamma\in\{0,1\}}$ are interface kinetic functions. We assume that these functions are
constant and strictly positive;
\item $\Delta\P_0^\textrm{pzc}$, $\Delta \P_1^\textrm{pzc}$ are respectively the outer and inner voltages of zero charge, $V$ is the applied potential.
\end{itemize}

Existence of a global weak solution for a system close to~\eqref{DPCM}--\eqref{CI} has been established in~\cite{CHLV-DCDS}. The main difference relies in the definition of the boundary conditions \eqref{CLDPCM.u2.0} and \eqref{CLDPCM.u2.1} for the electrons. Moreover the result is obtained under restrictive assumptions on the parameters, the physical sense of which being unclear. 
\rd
Let us highlight the misfit of model~\eqref{DPCM}--\eqref{CI} with respect to Onsager's reciprocal relation 
\cite{Onsager_1931_I,Onsager_1931_II} or its generalization beyond the linear setting~~\cite{Mie11}. 
Since inertia is not intended to play a role in our model, one expects the fluxes to be proportional to the driving forces, which can be decomposed into chemical and electrical contributions: 
\[
J_i=-\sigma_i\pa_x \xi_i, \quad\mbox{ for }i = 1,2, 
\]
where $\sigma_i >0$ denotes the mobility of the $i^\textrm{th}$ species (which may depend on $u_i$), and where 
\be\label{eq:xi_i}
\xi_i = v_i + z_i v_0, \quad\mbox{ for } i = 1,2, 
\ee
denotes the electrochemical potential, $v_i$ being the chemical potential of species $i$. 
The expressions for the fluxes in~\eqref{DPCM.u1} and \eqref{DPCM.u2} then suggest that 
\be\label{eq:DPCM.bulk.old}
\sigma_i = d_i u_i \quad\textrm{and}\quad v_i = \log\, u_i + \textrm{constant}, \quad\mbox{ for } i = 1,2. 
\ee
As a consequence,~\eqref{DPCM.u1} and  \eqref{DPCM.u2} are compatible with thermodynamics, provided that $u_i$ and $v_i$ 
are linked through Boltzmann statistics. 

On the other hand, one expects the boundary fluxes to depend on the electrochemical potential drop. Let us denote by $\xi_i^{\G}$ the electrochemical potential on the other side of the interface $\Gamma$ (\textit{i.e.} in the solution if $\Gamma = 0$ and in the metal if $\Gamma = 1$), while  $\xi_i(\Gamma)$ is the trace on $\Gamma$ of 
the electro-chemical potential $\xi_i$ defined in the oxide layer $(0,1)$. 
More precisely, one expects that 
\be\label{eq:Onsager.Gamma} 
J_i(\Gamma)\cdot \nu^\Gamma=r_i^\G g_i^\G(\xi_i(\Gamma)-\xi_i^{\G}), \qquad\mbox{ for }i = 1, 2, \mbox{ and }\G \in \{0,1\}, 
\ee
where $\nu^\G$ denotes the normal to $\Gamma$ ($\nu^0 = -1$ and $\nu^1 = 1$), while $r_i^\G$ is positive and 
possibly depends on $u_i$ and $g_i^\G$ is a non-decreasing function such that $g_i^\G(0) = 0$, so that 
\be\label{eq:g_i}
\forall y\in\R,\  y \ g_i^\G(y)  \ge 0, \quad \mbox{ for } i = 1,2, \mbox{ and } \G \in \{0,1\}.
\ee
As a consequence of~\eqref{eq:g_i}, boundary conditions of type~\eqref{eq:Onsager.Gamma} are dissipative in the 
sense that 
\be\label{DISS.CL}
J_i(\Gamma)\cdot \nu^\Gamma (\xi_i(\Gamma)-\xi_i^{\G}) \ge 0. 
\ee

Denoting 
\[
\kappa_1^0=2\sqrt{k_1^0 m_1^0},\qquad \kappa_1^1=2\sqrt{k_1^1 m_1^1},\qquad \xi_1^0=\log\dfrac{m_1^0}{k_1^0},\qquad
\xi_1^1=\log \dfrac{k_1^1}{m_1^1}+z_1 V,
\]
we can rewrite the boundary conditions for the cations~\eqref{CLDPCM.u1.0}and \eqref{CLDPCM.u1.1} as 
\rd
\[
J_1^\G \cdot \nu^\G = \kappa_1^\G\sqrt{u_1(\G)({\ov u}_1 -u_1(\G))}\sinh \lt(\ds\dfrac{1}{2}\lt(\log \dfrac{u_1(\G)}{{\ov u}_1 -u_1(\G)}+z_1 v_0(\G)-\xi_1^\G\rt)\rt).
\]
\bk
It appears then natural to define the chemical and electrochemical potentials of the cations by 
\be\label{eq:Blakemore}
v_1=\log \dfrac{u_1}{{\ov u}_1 -u_1}\quad \mbox{ and }\quad \xi_1=v_1+ z_1 v_0
\ee
in order to satisfy \eqref{eq:Onsager.Gamma} with 
\be\label{eq:r1g1}
r_1^\G = \kappa_1^\G \sqrt{u_1(\G)({\ov u}_1 -u_1(\G))}\mbox{ and }
g_1^\G(y) = \sinh \lt(\dfrac{y}{2}\rt).
\ee
In other words, the cations should rather obey a Blakemore statistics than Boltzmann statistics as suggested by~\eqref{eq:DPCM.bulk.old}. \bk

Similarly, we can rewrite the boundary condition~\eqref{CLDPCM.u2.0} as 
\[
-J_2(0)=\kappa_2^0\sqrt{u_2(0)}\sinh\lt( \dfrac{1}{2}(\log u_2(0) + z_2 v_0(0)-\xi_2^0)\rt)
\]
by setting  
\[
\kappa_2^0=2\sqrt{k_2^0 m_2^0},\qquad \xi_2^{0}=\log \dfrac{m_2^0}{k_2^0}.
\]
\rd 
Under this form, the electron flux $J_2(0)$ clearly enters the framework of~\eqref{eq:Onsager.Gamma} for
\be\label{eq:r20g20}
r_2^0 = \kappa_2^0 \sqrt{u_2(0)}\quad\mbox{ and }\quad
g_2^0(y) = \sinh \lt(\ds\dfrac{y}{2}\rt)
\ee
and 
\[
v_2 = \log(u_2)\quad\mbox{ and } \quad\xi_2= v_2 + z_2 v_0.
\]
Boltzmann statistics is encoded here, in accordance to what was prescribed in the bulk~\eqref{eq:DPCM.bulk.old}. 

Yet, it seems  impossible to recast the boundary condition~\eqref{CLDPCM.u2.1} for the electrons at the oxide/metal interface $\Gamma = 1$ into the framework~\eqref{eq:Onsager.Gamma}. This led us to propose the following modification of the boundary condition~\eqref{CLDPCM.u2.1}:
\be\label{CLDPCM.u2.1_new}
J_2(1)=m_2^1 u_2(1)-k_2^1 e^{z_2 (V-v_0(1))},
\ee
 which enters the framework of~\eqref{eq:Onsager.Gamma} by setting 
 \be\label{eq:r21g21}
 r_2^1 = m_2^1 u_2(1), \qquad \xi_2^1 = \log \dfrac{k_2^1}{m_2^1} + z_2 V  \quad \textrm{ and }\quad g_2^1(y) = 1 - e^{-y}. 
 \ee
 
Our concern related to the boundary condition~\eqref{CLDPCM.u2.1} having been solved by 
changing it into~\eqref{CLDPCM.u2.1_new}, the last misfit to be solved is related to the Blakemore statistics~\eqref{eq:Blakemore} for the cation which is not consistent with the bulk equation~\eqref{DPCM.u1}. On the other hand, the Butler-Volmer laws~\eqref{CLDPCM.u1.0} and \eqref{CLDPCM.u1.1} suggest some vacancy diffusion involving a nonlinear mobility $\sigma_1$, whereas interstitial diffusion corresponding to some linear mobility $\sigma_1$ has been prescribed so far in the bulk for the cations by \eqref{eq:DPCM.bulk.old}. Here, we suggest to fully adopt the vacancy diffusion process, yielding
\be\label{eq:vacancy_diff}
\sigma_1 = d_1\dfrac{u_1(\ov u_1 - u_1)}{\ov u_1},
\ee 
with $v_1$ being related to $u_1$ through the Blakemore statistics~\eqref{eq:Blakemore}.
With this choice, the diffusion remains linear, but the convection due to the electric potential becomes nonlinear:
\be\label{eq:J_1}
\partial_t u_1 + \partial_x J_1 = 0, \qquad J_1 = - \sigma_1 \partial_x (v_1+z_1 v_0)=-d_1(\partial_x u_1+ z_1\dfrac{u_1(\ov u_1 - u_1)}{\ov u_1}\partial_x v_0). 
\ee
For the electrons, we stick to band conduction leading to linear mobility $\sigma_2$ and 
to Boltzmann statistics:
\be\label{eq:band_cond}
\sigma_2 = d_2 u_2, \qquad v_2 = \log\, u_2. 
\ee
so that~\eqref{DPCM.u2} remains unchanged. 
\bk

\subsection{Main results and outline of the paper}

Our aim in this paper is to prove the existence of a weak solution under very general assumptions to the new DPCM introduced above,
\rd  
where~\eqref{DPCM.u1} has been changed into~\eqref{eq:J_1} with $\sigma_1$ defined by~\eqref{eq:vacancy_diff}, and where the oxide/metal interface condition~\eqref{CLDPCM.u2.1} has been modified into~\eqref{CLDPCM.u2.1_new}.  
\bk
In Section~\ref{sett}, we fix the mathematical setting: we recall the notations and the system of partial differential equations we consider; we also collect the general assumptions and give a weak formulation~\eqref{eq:P} to the model. The main result, namely the global in time existence of a solution, is then stated in Theorem~\ref{teoex}. Section~\ref{sec.apriori} is devoted to physically motivated estimates and more general a priori estimates on a solution to~\eqref{eq:P}. Before dealing with~\eqref{eq:P}, we introduce a family of regularized problems~\eqref{eq:PM} (with $M>0$) in Section~\ref{sec.PM} and we prove their solvability as stated in Proposition~\ref{expm}. Finally, in Section~\ref{sec.bounds}, we establish some lower and upper bounds for the solution to~\eqref{eq:PM} which do not depend on the regularization level $M$ (when sufficiently large). These estimates lead to the existence of a solution to~\eqref{eq:P}.

\section{Mathematical setting \rd and main result \bk}\label{sett}
In this section we give the precise setting of the problem we are concerned with. 
\subsection{Notation}
In addition to the notation already introduced in Section~\ref{sec.DPCMvDPCM}, we denote by 
$\ov u_2$ the reference occupancy for electrons (equal to $1$ in practice), 
and by $u_0$ the total charge density in the oxide layer. 
The outer electro-chemical potentials of iron cations and electrons at the interfaces are denoted by $\xi_i^\G,\ i=1,2,\ \G\in \{0,1\}$, which are assumed to not depend on time. 
Finally, 
\[
f^0 = \dfrac{\lambda^2}{\alpha_0} \Delta \P_0^\textrm{pzc}, 
\qquad 
f^1 = \dfrac{\lambda^2}{\alpha_1} (V - \Delta\P_1^\textrm{pzc})
\quad\ \textrm{ and }\ \quad
\beta^\G = \dfrac{\lambda^2}{\alpha_\Gamma},
\quad\mbox{ for } 
\G\in \{0,1\}
\]
 are some given values related respectively to the interface potentials and to the differential capacitance of the boundaries.
Then, we consider the corrosion model as a system of partial differential equations whose unknowns are the charge densities $(u_0,u_1,u_2)$ and the electrical/chemical potentials $(v_0,v_1,v_2)$. It writes in  $(0,1)\times (0,+\oo)$, for $i=1,2$:
\begin{subequations}\label{vDPCM}
\begin{gather}
\pa_t u_i +\pa_x J_i=0, \label{vDPCM.dens}\\
\quad\text{with }\quad J_i=-\sigma_i(v_i)\pa_x \xi_i,\quad \xi_i=z_i v_0 +v_i,\label{vDPCM.Ons}\\
-\lambda^2\pa_{xx}v_0=u_0, \label{vDPCM.pot}\\
\quad\text{with }\quad u_0=\sum_{i=1,2} z_iu_i+\rho_{hl}.
\end{gather}
\end{subequations}
The boundary conditions are defined on $\Gamma\times(0,+\oo)$ with $\Gamma\in \{0,1\}$ by:
\begin{subequations}\label{vDPCM.CL}
\begin{gather}
J_i\cdot \nu^\G=r_i^\G(v_i)g_i^\G(\xi_i-\xi_i^{\G}),\label{vDPCM.CL.dens}\\
\lambda^2 \pa_x v_0 \cdot \nu^\G +\beta^\G v_0=f^\G .\label{vDPCM.CL.pot}
\end{gather}
\end{subequations}
Moreover, as it has been explained in the introduction, we may consider the following relations between the densities and the chemical potentials:
\be\label{vi}
v_1=\displaystyle \log\lt(\dfrac{u_1}{\ov{u}_1-u_1}\rt)\quad \text{ and }\quad v_2=\log\lt(\dfrac{u_2}{\ov{u}_2}\rt),
\ee
or in an equivalent manner:
\be\label{ui}
u_1=\ov{u}_1\dfrac{e^{v_1}}{1+e^{v_1}}\quad \mbox{ and }\quad u_2=\ov{u}_2 e^{v_2}.
\ee
This corresponds to a Blakemore statistics for the cations and to a Boltzmann statistics for the electrons.

Each mobility $\sigma_i$ for $i=1,2$ has been given in~\eqref{eq:vacancy_diff} and~\eqref{eq:band_cond} as a function of $u_i$, but it can finally be considered  as a function of $v_i$ as it will be detailed below. Similarly, the function $r_i^\Gamma$ appearing in the boundary conditions introduced in Section~\ref{sec.DPCMvDPCM} can be written as functions of the chemical and electro-chemical potentials of the form~\eqref{vDPCM.CL.dens}.

\subsection{Assumptions on the data}
We give here in details the hypotheses we assume throughout the paper.
\vskip 1mm
\begin{enumerate}[(H$_1$)]\itemsep = 1mm
\item  $\lambda>0$ and for $\G\in \{0,1\}$, $\beta^\G>0$ and $f^\G \in \mathbb{R}$.
\item The densities are related to the chemical potentials through 
\[u_i={\ov u}_i e_i(v_i)\quad \mbox{ for }i=1,2,\]
where the functions $e_i$ are defined on $\R$ by
\[
e_1(z)= \dfrac{e^z}{1+e^z}\quad\ \mbox{ and }\ \quad e_2(z)=e^z.
\]
\item The mobilities are related to the chemical potentials through $\sigma_i(v_i)=d_i{\ov u}_i e'_i(v_i)$ with $d_i>0$ for $i=1,2$. This means that 
\[
\sigma_1(z)=d_1{\ov u}_1 \dfrac{e^z}{(1+e^z)^2}\quad\ \mbox{ and }\ \quad \sigma_2(z)= d_2{\ov u}_2e^z.
\]
\item The positive functions $r_i^\G:\R\rw (0,+\oo)$, for $i=1,2$ and $\G\in\{0,1\}$, are defined by
\[
\lt\{\begin{array}{ll}
\displaystyle
r^\G_1(v_1)=\kappa^\G_1 {\ov u}_1\dfrac{e^{\dfrac{v_1}{2}}}{1+e^{v_1}},\ \G\in \{0,1\},\\[4mm]
\displaystyle
r^0_2(v_2)=\kappa^0_2 \sqrt{{\ov u}_2} e^{\dfrac{v_2}{2}},\\[4mm]
\displaystyle
r^1_2(v_2)=\kappa^1_2 {\ov u}_2 e^{v_2},
\end{array}\rt.
\]
with $k_i^\G$ positive constants. The functions $g_i^\G$ for $i=1,2$ and $\G \in \{0,1\}$ are respectively defined by 
\be\label{eq:g_i.def}
\lt\{\begin{array}{ll}
\displaystyle
g^\G_1(y)=\sinh\lt(\dfrac{y}{2}\rt),\ \G\in\{0,1\},\\[4mm]
\displaystyle
g^0_2(y)=\sinh\lt(\dfrac{y}{2}\rt),\\[4mm]
\displaystyle
g^1_2(y)=1-e^{-y}.
\end{array}\rt.
\ee
\rd 
Note that all the functions $g_i$ are increasing and vanish at $0$ so that~\eqref{eq:g_i} holds true. 
\bk
\item \rd The initial profiles $u_i^\textrm{in}$, $0 \leq i \leq 2$, are such that 
\[
u_0^\textrm{in}=\displaystyle \sum_{i=1,2}z_i u_i^\textrm{in}+~\rho_\textrm{hl}, 
\]
and such that the corresponding chemical potentials are bounded, \textit{i.e.}, 
\[
v_i^\textrm{in} = e_i^{-1}\lt(\dfrac{u_i^\textrm{in}}{\ov u_i}\rt) \in L^\oo(0,1), \quad \mbox{ for }i = 1,2. 
\]
This is equivalent to requiring that $u_1^\textrm{in}$ is bounded away from $0$ and $\ov u_1$, 
whereas $u_2^\textrm{in}$ is bounded away from $0$. 
\bk
\end{enumerate}

\subsection{Notion of solution}
In order to give a variational formulation to our system, we introduce the spaces \[V=H^1(0,1)\times H^1(0,1)\times H^1(0,1)\quad \text{ and }\quad H=H^1(0,1)\times L^2(0,1)\times L^2(0,1),\]
equipped with their standard norms. \rd The system is to be regarded as an initial boundary value problem for the main unknown vector $(v_0, v_1, v_2)$ of potentials and the corresponding dual vector $(u_0, u_1, u_2)$ of densities. \bk Then we introduce the operators $E:H \rw H^*$ and $A:H \times V\rw V^*$ (where $X^*$ is the topological dual space of $X$) defined by:
\[\label{opEA}
\begin{array}{ll}
\displaystyle
\lt< Ew,\tilde{v}\rt>=\int_0^1 \lt(\lambda^2\pa_x w_0\pa_x \tilde{v}_0 +\sum_{i=1,2}{\ov u}_i e_i(w_i)\tilde{v}_i\rt) dx +\sum_{\G\in \{0,1\}}\lt[ (\beta^\G w_0-f^\G)\tilde{v}_0\rt](\G),\\[6mm]
\displaystyle
\lt< A(w,v),\tilde{v}\rt>= \sum_{i=1,2}\int_0^1 \sigma_i( w_i) \pa_x \xi_i \pa_x \tilde{\xi}_i dx +\sum_{i=1,2}\sum_{\G\in \{0,1\}}\lt[ r_i^\G(w_i)g_i^\G(\xi_i-\xi_i^{\G})\tilde{\xi}_i\rt](\G),
\end{array}
\]
where $w=(w_0, w_1, w_2)\in H$,  $v=(v_0, v_1, v_2)\in V$,  $\tilde{v}=(\tilde{v}_0, \tilde{v}_1, \tilde{v}_2)\in V$,    
\[\xi_i=z_i v_0+v_i\quad \text{ and }\quad \tilde{\xi}_i= z_i \tilde{v}_0 +\tilde{v}_i,\ \mbox{ for } i=1,2.\]
Then, the weak formulation of problem~\eqref{vDPCM}, \eqref{vDPCM.CL} reads
\be\label{eq:P}\tag{$P$}
\lt\{\begin{array}{ll}
\displaystyle
\text{Find }(u,v) \text{ such that}\\[3mm]
\displaystyle
u \in H^1_\textrm{loc}(\mathbb{R}_+;V^*),\quad  v \in L^2_\textrm{loc}(\mathbb{R}_+;V) \cap L^{\oo}_\textrm{loc}(\mathbb{R}_+\times [0,1]),\\[3mm]
\displaystyle
\dot u(t)+A(v(t),v(t))=0, \quad u(t)=Ev(t),\mbox{ for a.e. } t\in \R_+, \quad u(0)=u^{\text{\rm in}}.
\end{array}\rt.
\ee
\begin{oss}\label{rmk:charge}
Standard arguments show that if $(u,v)$ is a pair of smooth functions, then $(u,v)$ solves~\eqref{eq:P} if and only if it satisfies~\eqref{vDPCM}--\eqref{ui}. More precisely, on the one hand, the condition $u=Ev$ expresses both the relations $u_i=\bar u_i e_i(v_i)$ for $i=1,2$ as well as the boundary value problem~\eqref{vDPCM.pot}, \eqref{vDPCM.CL.pot} relating the electrostatic potential $v_0$ to the charge density. On the other hand, $\dot u + A(v,v)=0$ is a weak formulation of the advection diffusion equations~\eqref{vDPCM.dens}, \eqref{vDPCM.Ons}, \eqref{vDPCM.CL.dens}. Moreover, if $(u,v)$ solves~\eqref{eq:P}, testing $\dot u+A(v,v)=0$ against the functions of the form $\tilde w=(y_0,-z_1 y_0, -z_2y_0)$ for $y_0\in C^1_c(0,1)$, we obtain
\[
\partial_t\lt(u_0-\sum_{i=1,2}z_i u_i\rt)=0\quad\text{in }\d'(\R_+\times (0,1)).
\]
Using the initial condition, we get that any solution $(u,v)$ to~\eqref{eq:P} satisfies
\[
u_0=\sum_{i=1,2}z_i u_i + \rho_\text{\rm hl}.
\]
\end{oss}
The aim of this paper is to prove an existence result for problem~\eqref{eq:P}:
\begin{teo}\label{teoex}
Under assumptions~{\rm(H$_1$)--(H$_5$)}, there exists at least one solution to problem~\eqref{eq:P}.
\end{teo}

\section{Free energy and a priori estimates}\label{sec.apriori}
Before going into the details of the proof of Theorem~\ref{teoex}, we give in this section an explicit expression for the total free energy, assuming that a solution to problem~\eqref{eq:P} exists. \rd Then we show that our model is thermodynamically consistent in the sense that the free energy of a solution decreases over time. Furthermore, this leads to some \textit{a priori} estimates for the solution which turn to be crucial in the proof.
\bk  
To this end we follow some ideas from~\cite{GG_1996}. Define 
\[
\displaystyle\f_i(v)=\int_0^v e_i(y)dy\quad \textrm{ for }v \in \mathbb{R}, \; i =1,2,
\]
and 
\[
\p_i(u)=\int_{e_i(0)}^u e^{-1}_i(z)dz\quad \textrm{for $u \in [0,\ov u_1]$ if $i=1$ and $u\ge 0$ if $i  = 2$.}
\]
 For the specific choice of nonlinearities of (H$_2$), this provides
\[
\begin{cases}
\f_1(v)=\log(1+e^v)-\log 2,  \\ 
\p_1(u)=u\log u+(1-u) \log(1-u)+\log 2,\end{cases}
\qquad\text{and}\qquad
\begin{cases}
\f_2(v)=e^v-1,\\
\p_2(u)=u\log u -u +1.
\end{cases} 
\]
The non-negative convex functions $\p_i$ satisfy $\psi_1(1/2)=\psi_2(1)=0$ and are extended by $+\oo$ outside of their domain of definition.

We define the Landau free energy of $v \in V$ by
\[
\F(v)=\int_0^1 \lt(\dfrac{\lambda^2}{2}|\pa_x v_0|^2+ \sum_{i=1,2} {\ov u}_i \f_i(v_i)\rt)dx + \sum_{\G\in \{0,1\}}\lt(\dfrac{\beta^\G}{2}v_0^2-f^\G v_0\rt)(\G).
\]
Its conjugate, the Helmholtz free energy, is defined for $u \in V^*$ by
\[
\P(u)=\sup_{v\in V}\lt\{\lt<u,v\rt>-\F(v)\rt\},
\]
where $\lt<\, ,\, \rt>$ denotes the standard scalar product in $L^2(0,1)$.

Note that if $\f_i(v_i)$, $i=1,2$, are not in $L^1$, the values $\F(v)$ and $\P(u)$ are interpreted as $+\oo$. Moreover, $\F$ and $\P$ turn out to be strictly convex functionals and $\F(0)=0$, hence for every $v \in V$ the subdifferential of $\F$ contains at most one element. More precisely one checks that $\pa \F=\{Ev\}$.
Then standard computations show that
\be\label{FE}
\P(u)=\int_0^1 \lt(\dfrac{\lambda^2}{2}|\pa_x v^*_0|^2+\sum_{i=1,2} {\ov u}_i \psi_i\lt(\dfrac{u_i}{{\ov u}_i}\rt)\rt)dx +\sum_{\G\in \{0,1\}}\lt[ \dfrac{\beta^\G}2|v^*_0|^2\rt](\G),
\ee
with $v_0^*$ solving the Poisson equation with Robin-Fourier boundary conditions
\[
\lambda^2\int_0^1 \pa_x v^*_0 \pa_x \tilde{v}\,dx + \sum_{\G\in \{0,1\}}\lt[\lt({\rd\beta^\G}v^*_0-f^\G \rt)\tilde{v}\rt](\G)=\int_0^1 u_0 \tilde{v}\,dx, \quad \forall \tilde{v} \in V.
\]

\rd
The Helmholtz free energy of an isolated system is expected to be a Lyapunov functional. Here, we have fluxes across the interfaces $\G \in \{0,1\}$ which may contribute positively to the variations of $\P(u)$. In order to get an isolated system (and hence a  Lyapunov functional), then one introduces some very elementary model for the charge carriers leaving the oxide. More precisely, we assign the energy $\xi_1^0$ (resp. $\xi_1^1$) to each unit of cations entering the solution (resp. the metal) from the oxide. Similarly the energy of one unit of electrons leaving the oxide is set to $\xi_2^\Gamma$, $\Gamma \in \{0,1\}$.  Therefore the free energy associated to elements leaving the oxide layer to the solution and the metal are respectively defined by 
\[
\P^\G(t) = \sum_{i=1,2}\int_0^t \sum_{\G\in \{0,1\}}\lt[ (J_i\cdot \nu^\G)\xi_i^\G\rt](\G) d\tau, 
\qquad \G \in \{0,1\}.
\]
Finally, 
the total free energy $\P^\textrm{tot}$ is  given by
\be\label{psitot}
\P^\textrm{tot}(t)= \P(u(t)) + \sum_{\G\in \{0,1\}} \P^\G(t).
\ee

We prove in the following proposition the decay of the total free energy $\P^\textrm{tot}$ over time. This estimate, which encodes the second principle of thermodynamics, is the key \textit{a priori} estimate on which our analysis builds. Here and in what follows, the space $H^1(0,1)$ is equipped with the norm 
\be\label{eq:H1}
{\|w\|}_{H^1(0,1)} = \lt(\dfrac12 \int_0^1 |\pa_x w|^2 dx + \dfrac12\sum_{\G \in \{0,1\}} |w(\G)|^2 \rt)^{1/2}.
\ee
\bk
\begin{prop}\label{subd}
\rd 
Assume that~{\rm(H$_1$)--(H$_5$)} hold and let $(u,v)$ be a solution of problem~\eqref{eq:P}. For $0\le s\le t$,
\begin{align*}
\P^\textrm{tot}(t)-\P^\textrm{tot}(s)=&-\sum_{i=1,2}\int_s^t\int_0^1 \sigma_i(v_i) |\pa_x \xi_i|^2 dx d\tau\\[3mm]
 &-\sum_{i=1,2}\int_s^t\sum_{\G\in \{0,1\}}\lt[ r_i^\G(v_i)g_i^\G(\xi_i-\xi_i^{\G})(\xi_i-\xi_i^{\G})\rt](\G) d\tau. 
\end{align*}
\rd In particular, there holds \bk 
\be\label{pdec}
\P^\text{\rm tot}(t)\le \P^\text{\rm tot}(s)\le \P(0)< \oo.
\ee
As a consequence, {\rd there exists $c_T>0$ depending on the data of the continuous problem and on $T$ such that}
\be\label{vulu}
\|v_0\|_{L^{\oo}((0,T); H^1(0,1))}+\sum_{i=1,2} \rd \lt\| \psi_i\lt(\dfrac{u_i}{\ov u_i}\rt)\rt\|_{L^{\oo}((0,T); L^1(0,1)) \bk }\le c_T.
\ee
In particular this implies that $u_i/\ov u_i$ belongs to the domain of $\psi_i$ for almost every $(x,t)$, \textit{i.e.} $0 \leq u_1 \leq \ov u_1$ and $u_2 \geq 0 $.
\end{prop}
\begin{proof}
Let $(u,v)$ be a solution to problem~\eqref{eq:P}. Then for almost every $t \in \mathbb{R}_+$, \[u(t)=Ev(t)\in \pa \F(v(t)),\] which is equivalent to 
\[v(t)\in \pa \P(u(t))\] 
\rd since $\F$ and $\P$ are the Legendre transform of each other. \bk
Thus, for $0\le s\le t$,
\[\P(u(t))-\P(u(s))=\int_s^t \lt< \dot u(\tau),v(\tau)\rt> d\tau=-\int_s^t \lt< A(v(\tau),v(\tau)),v(\tau)\rt> d\tau\]
and consequently \rd due to the definition~\eqref{psitot} of $\P^\textrm{tot}$ \bk
\begin{align}\label{decen}
\P^\textrm{tot}(t)-\P^\textrm{tot}(s)
=&-\sum_{i=1,2}\int_s^t\int_0^1 \sigma_i(v_i) |\pa_x \xi_i|^2 dx d\tau 
\\
&-\sum_{i=1,2}\int_s^t\sum_{\G\in \{0,1\}}\lt[r_i^\G(v_i)g_i^\G(\xi_i-\xi_i^{\G})(\xi_i-\xi_i^{\G})\rt](\G)d\tau\le 0, \nonumber
\end{align}
\rd thanks to~\eqref{eq:g_i}. \bk
Also note that $\P^\textrm{tot}(t)$ is finite for every $t \in \mathbb{R}_+$ when $(u,v)$ is a solution of~\eqref{eq:P} \rd since $\P^\textrm{tot} (0)= \P(u^\textrm{in})$ is finite thanks to Assumption~(H$_5$), \bk then we obtain the validity of~\eqref{pdec}.
{\rd Then one readily checks (calculations are detailed in the time discrete case in Section~\ref{sec.PM}, see~\eqref{eq:PMn.4}) that 
the Helmholtz free energy corresponding to the oxide layer only remains bounded, but not uniformly w.r.t. time, \textit{i.e.} 
$\P(u(t)) \le c_T$ for $t \in [0,T]$.}
This implies~\eqref{vulu} in view of~\eqref{FE}.
\end{proof}

\section{Existence result for a regularized problem~\eqref{eq:PM}}\label{sec.PM}
In order to obtain the existence result for problem~\eqref{eq:P} stated in Theorem~\ref{teoex}, we first introduce a regularized problem. It is denoted by~\eqref{eq:PM} and it is obtained from~\eqref{eq:P} by cutting off all the nonlinearities \rd applied to the chemical potentials $v_1,v_2$ \bk at a certain level $M$. This section is devoted to the solvability of such a regularized problem, which is given in Proposition~\ref{expm}. Via a discretization of time, we construct a sequence of approximate solutions to problem~\eqref{eq:PM}. Then, accurate a priori estimates and compactness arguments ensure the existence of at least one solution to problem~\eqref{eq:PM}.

In the next section, for such a solution, several a priori estimates will be proved independently on the level $M$. Consequently, a solution to~\eqref{eq:PM} will be also a solution to~\eqref{eq:P} when choosing the level $M$ sufficiently large.

This technique was originally introduced in a series of seminal papers by Gajewski and Gr\"{o}ger~\cite{GG_1986,GG_1989, GG_1990}. The main differences with respect to those works consist in a different expression of the total free energy and in the presence of nonlinear Robin boundary conditions we have to deal with.\\
Let $M>0$ be a fixed \rd parameter  chosen large enough to ensure \bk
\be\label{M}
M\ge \max_{i=1,2}\|{\rd v}_i^{in}\|_{L^{\oo}(0,1)}.
\ee
We introduce the usual truncation function at level $M$ given by 
\[
T_M(z)= \rd \max(-M, \min(M,z)) = \bk 
\begin{cases}
~~~ z&\text {if }|z|\le M,\smallskip\\
~\pm M &\text{if }\pm z>M.
\end{cases}
\]
We {\rd also} define by $E_M: H \rw H^*$ and $A_M:H \times V\rw V^*$ the operators defined by 
\[
\lt< E_M v,\tilde{v}\rt> =\int_0^1 \lt(\lambda^2\pa_x v_0\pa_x \tilde{v}_0 +\sum_{i=1,2}{\ov u}_i e_i(T_M v_i)\tilde{v}_i\rt) \, dx+\sum_{\G\in \{0,1\}}\lt[(\beta^\G v_0-f^\G)\tilde{v}_0\rt](\G)\]
and
\begin{equation}
\lt< A_M(w,v),\tilde{v}\rt> = \int_0^1 \sum_{i=1,2}\sigma_i(T_M w_i) \pa_x \xi_i \pa_x \tilde{\xi}_i\,  dx  +\sum_{i=1,2}\sum_{\G\in \{0,1\}}\lt[r_i^\G(T_M w_i)g_i^{\G,\mu}(\xi_i-\xi_i^{\G})\tilde{\xi}_i\rt](\G), \label{eq:AM}
\end{equation}
where $\xi_i=z_i v_0+v_i$ and $\tilde{\xi}_i= z_i \tilde{v}_0 +\tilde{v}_i$, $i=1,2$.
{\rd For technical reasons that will appear later on in the proof, we also have to modify the nonlinear Robin boundary conditions~\eqref{eq:g_i.def}}. More precisely, for $i=1,2$ and $\mu>0$ (yet another parameter to be tuned later on), $g_i^{\G,\mu}$ denotes the following {\rd approximation of $g_i^\G$}:
\be\label{eq:giGmu}
g_i^{\G,\mu}(\xi)=
\begin{cases}
g_i^\G(\xi)&\text{if }|\xi|\le \mu,\medskip\\
g_i^\G(\pm \mu)+(\xi\mp\mu) (g^\G_i)'(\mu)&\text{if }\pm\xi> \mu.
\end{cases}
\ee
where the functions $g_i$ are the ones introduced in (H$_4$). {\rd The functions $g_i^{\G,\mu}$ turns out to be Lipschitz continuous functions coinciding with $g_i$ on the interval $[-\mu;\mu]$ and being linear outside of it.  Since  $g_1^{\G}$ and $g_2^0$ are even functions, $g_1^{\G,\mu}$ and $g_2^{0,\mu}$ belong to $C^1(\R)$, while $g_2^{1,\mu}$ is merely $C^{0,1}(\R)$.}  
Then our regularized problem writes
\be\label{eq:PM}\tag{$P_M$}
 \lt\{\begin{array}{ll}
\displaystyle
\text{Find }(u,v) \text{ such that}\\[3mm]
\displaystyle
u \in H^1_\textrm{loc}(\mathbb{R}_+;V^*),\quad  v \in L^2_\textrm{loc}(\mathbb{R}_+;V),\\[3mm]
\displaystyle
\dot u(t)+A_M(v(t),v(t))=0, \quad u(t)=E_Mv(t),\mbox{ for a.e. } t\in \R_+,\quad u(0)=u^{\text{\rm in}}.
\end{array}\rt.
\ee
The next result provides the existence of (at least) one solution, \rd still denoted by $(u,v)$, 
\bk to problem~\eqref{eq:PM}.

\begin{prop}\label{expm}
Under assumptions {\rm(H$_1$)--(H$_5$)} and 
 if 
\be\label{cond_mu}
 \mu\le M - \max_{i=1,2} \lt( |z_i| \cter{cte:v0} + |\xi_i^\G|\rt)
\ee
with $\ctel{cte:v0}$ introduced hereafter in~\eqref{eq:bound.v0}, 
then there exists a solution $(u,v)$ to problem~\eqref{eq:PM}.
Moreover, $v_0$ satisfies
\be\label{eq:v0bound}
{\| v_0\|}_{L^\oo(I;H^1(0,1))}\leq \cter{cte:v0}
\ee
and there exists $\ctel{cte:v0.Lip}>0$ depending on  $T$ and on the data 
but not on $M$ such that 
\be\label{eq:v0.Lip}
{\| \pa_x v_0 \|}_{L^\oo((0,1) \times (0,T))} \le \cter{cte:v0.Lip}.
\ee
\end{prop}
The {\rd remainder of this section is devoted to the proof of Proposition~\ref{expm}. We proceed  in four steps. First we construct a sequence of time discrete approximations $(P_{M,n})_{n \ge 1}$ of  problem~\eqref{eq:PM}, the solutions of which being denoted by $(u_n,v_n)_{n\ge 1}$}. Then we derive {\rd some} estimates for such solutions. In the third step we invoque compactness arguments to pass to the limit as $n\rw +\oo$ and recover a time continuous solution $(u,v)$ to model~\eqref{eq:PM}. The last step is devoted to the proof of the regularity estimate~\eqref{eq:v0.Lip}.
\vskip 2mm

{\textbf{Step 1.}} 
Let us fix $T\in [0,+\oo)$ a finite but arbitrary time horizon, and set $I=(0,T]$. For $n\ge1$, 
we define the time step $k_n=T/n$ and the time intervals $I_n^j= \big((j-1)k_n, j k_n\big]$, for $j=1,....,n$.

Given a Banach space $X$, we denote by $B_n(I;X)$ the space of functions $u: (0,T]\rw X$ that are constant on each of the intervals $I_n^j$, $1 \le j \le n$. If $u \in B_n(I;X)$, we define $u^j\in X$ for $1\le j\le n$ by $u=\sum u^j \un_{I_n^j}$. We introduce two mappings $\Delta_n$ and $\tau_n$ from $B_n(I;V^*)$ into itself defined by:
\[
(\Delta_n u)^j=\dfrac{1}{k_n}(u^j-u^{j-1}), \qquad (\tau_n u)^j=u^{j-1}, \qquad 1 \le j \le n,
\]
where $u^0 ={\rd  u^\textrm{in}}$ is the initial datum. We consider the discrete version of problem~\eqref{eq:PM} corresponding to the time step $k_n$ given by
\be\label{eq:PMn} \tag{$P_{M,n}$}
\Delta_nu_n +A_M(\tau_n v_n,v_n)=0, \qquad u_n=E_M v_n,\quad v_n \in B_n(I;V).
\ee
It can be written as
\be\label{pmn}
\dfrac{1}{k_n}(u_n^j-u_n^{j-1})+A_M(v_n^{j-1},v^j_n)=0, \quad u^j_n=E_M v^j_n,\quad 1 \le j \le n,\quad u_n^0=u^{\text{\rm in}}.
\ee
{\rd Thanks to arguments similar to those of Remark~\ref{rmk:charge}, there holds 
\be\label{eq:charge.n}
u_{n,0} = \sum_{i = 1,2} z_i u_{n,i} + \rho_\text{\rm hl}. 
\ee
The following lemma is about the well-posedness of problem $(P_{M,n})$.}
\begin{lemma}\label{expmn}
Under assumptions {\rm(H$_1$)--(H$_5$)}, for any $M>0$, for every $n\ge 1$, there exists a unique solution $(u_n,v_n)$ to the problem $(P_{M,n})$.
\end{lemma}
\begin{proof}
We use some known results on monotone operators~\cite[Corollaire 17]{Brezis6566} (see also~\cite{Lions69, Brezis73}). Let us fix $y \in V$ and define the operator $F: V \rw V^*$ by
\[F(v) = \dfrac{1}{k_n}E_M v + A_M(y,v).\]
If such operator $F$ is strongly monotone, \textit{i.e.} if there exists $\alpha >0$ such that 
\[
\lt< F(v) - F(w), v-w \rt> \ge \alpha \| v-w \|^2, \qquad \forall \, v, w \in V, 
\]
\bk
then all the equations~\eqref{pmn} are uniquely solvable, considered as equations with respect to $v_n^j$ for given $v_n^{j-1}$ (and $u_n^{j-1}$). This gives the unique solvability of our problem $(P_{M,n})$.

Let us check the strong monotonicity of $F$. Let $v,w\in V$, we compute
\begin{align*}
 \lt< F(v)-\rt.&F(w),\lt.v-w\rt>\\
=& \dfrac{1}{k_n}\int_0^1 \bigg(\lambda^2 |\pa_x (v_0-w_0)|^2 + 
 \sum_{i= 1,2} {\ov u}_i (v_i-w_i)[e_i(T_M v_i)-e_i(T_M w_i)]\bigg)\, dx\\
&+\dfrac{1}{k_n}\sum_{\G\in \{0,1\}}\lt[\beta^\G (v_0-w_0)^2\rt](\G)
+\int_0^1 \sum_{i=1,2} \sigma_i(T_M y_i)|\pa_x(z_i(v_0-w_0)+(v_i-w_i))|^2\, dx\\
&+\sum_{i=1,2}\sum_{\G\in \{0,1\}}\Big\{ r_i^\G(T_M y_i)(z_i(v_0-w_0)+(v_i-w_i)) \\
&\phantom{a}\hspace{4.5cm}\times \lt[g_i^{\G,\mu}((z_iv_0+v_i)-\xi_i^\G)-g_i^{\G,\mu}((z_i w_0+w_i)-\xi_i^\G)\rt]\Big\}(\G).
\end{align*}
Using the monotonicity of the functions $e_i$ together with the fact that 
\[(x_1-x_2)(g_i^{\G,\mu}(x_1)-g_i^{\G,\mu}(x_2))\ge |x_1-x_2|^2, \qquad  x_1, x_2 \in \mathbb{R},\] 
we get:
\begin{align*}
\lt< F(v)-F(w),v-w\rt> \ge \; & 
\dfrac{1}{k_n}\int_0^1 \lambda^2 |\pa_x (v_0-w_0)|^2 dx+\dfrac{1}{k_n} \sum_{\G\in \{0,1\}}\lt[\beta^\G (v_0-w_0)^2\rt](\G)
\\&
+\sum_{i=1,2}\int_0^1 \sigma_i(T_M y_i)|\pa_x(z_i(v_0-w_0)+(v_i-w_i))|^2dx
\\
&+\sum_{i=1,2} \sum_{\G\in \{0,1\}}\lt[r_i^\G(T_M y_i)|z_i(v_0-w_0)+(v_i-w_i)|^2\rt](\G).
\end{align*}
Moreover, defining $c>0$ as $c=\min(\lambda^2, \min_{\Gamma\in\{0,1\}} \beta^\G)$ and using that the functions $y \mapsto \sigma_i(T_M y)$ and $y \mapsto r^\G_i(T_M y)$ are  bounded from below by some positive constant $c_M$  depending only on the data of the continuous problem and on $M$,
we deduce:
\be\label{preuve:expmn}
\lt< F(v)-F(w),v-w\rt>\ge \dfrac{c}{k_n}\|v_0-w_0\|^2_{H^1(0,1)}+ c_M \sum_{i=1,2} \|z_i(v_0-w_0)+(v_i-w_i)\|^2_{H^1(0,1)}.
\ee
For $i=1,2$, we notice that the following alternative holds: 
\begin{itemize}
\item either $\|z_i(v_0-w_0)\|_{H^1(0,1)}\leq \dfrac{1}{2} \|v_i-w_i\|_{H^1(0,1)}$, which implies, by triangular inequality,
$$
\|z_i(v_0-w_0)+(v_i-w_i)\|^2_{H^1(0,1)}\geq \dfrac{1}{4} \|v_i-w_i\|^2_{H^1(0,1)},
$$
\item or $\|z_i(v_0-w_0)\|_{H^1(0,1)}>\dfrac{1}{2} \|v_i-w_i\|_{H^1(0,1)}$, 
so that 
$$
 \|v_0-w_0\|^2_{H^1(0,1)}\geq \dfrac{1}{4 |z_i|^2}\|v_i-w_i\|^2_{H^1(0,1)}.
$$
\end{itemize}
Therefore, we obtain
$$
\lt< F(v)-F(w),v-w\rt>\ge \dfrac{c}{2k_n}\|v_0-w_0\|^2_{H^1(0,1)}+
\sum_{i=1,2} \min (\dfrac{c_M}{4}, \dfrac{c}{16k_n |z_i|^2})\|v_i-w_i\|^2_{H^1(0,1)}.
$$
This proves the strong monotonicity of $F$ and therefore the lemma.
\end{proof}
\vskip 2mm


{\textbf{Step 2.}} With  the sequence ${(u_n,v_n)}_{n\ge 1}$ at hand, we derive some estimates, uniform w.r.t. $n$.
\begin{lemma}\label{buvn}
Assume that assumptions~{\rm(H$_1$)--(H$_5$)} hold true. Let $n\ge 1$, and let $(u_n,v_n)$ be the unique solution to problem $(P_{M,n})$ given by Lemma~\ref{expmn}. We define 
\[
U_n(t)=u^{\text{\rm in}} + \int_0^t (\Delta_n u_n)(s)\, ds, \quad 0 \le t \le T,
\]
with $\Delta_n u_n(s)=u_n(s)-u_n(s-k_n)$ for $s\in(k_n,T]$ and $\Delta_n u_n=u_n(s)-u^{\text{\rm in}}=u_n^1-u^{\text{\rm in}}$ for $s\in(0,k_n]$ ($U_n$ is the piecewise affine extension on $I$ of $jk_n \mapsto u_n^j$).

There exists $c_{M,T}\ge0$ depending on $M$ and $T$ but not on $n$ such that 
\[
\sup_{n\in \mathbb{N}}\lt\{\|v_n\|_{L^2(I;V)}+\|\Delta_n u_n\|_{L^2(I;V^*)}+\|U_n\|_{C(I;V^*)}\rt\} \le c_{M,T}.
\]
\end{lemma}
\begin{proof}
Along the proof $c\ge0$ is a constant that may depend on the data but not on $n$ or $M$ and whose value may change from line to line. Similarly we denote by $c_M$, $c_T$, $c_{M,T}$ the constants that may depend on $M$, $T$ or both.\\
Let us set
\[
\f_{i,M}(v) = \int_0^v e_i(T_M y)dy, \quad v \in \R, 
\]
and then let us define the functionals $\F_M$ and $\P_M$ by
\[
\F_M(v)=\int_0^1 \lt(\dfrac{\lambda^2}{2}|\pa_x v_0|^2 + \sum_{i=1,2} {\ov u}_i \f_{i,M}(v_i)\rt)dx + \sum_{\G\in \{0,1\}}\lt[\dfrac{\beta^\G}{2}v_0^2-f^\G v_0\rt](\G),\ \text{ for } v \in V,
\]
and
\[
\P_M(u)=\sup_{v\in V}\{\lt< u,v\rt>-\F_M(v)\},\quad \text{ for } u \in V^*.
\]
Notice that $\F_M$ is continuous, convex and coercive.
\rd
Testing~\eqref{pmn} with $v_n^j$, for $1 \le j \le n$, provides 
\be\label{eq:est.PMn.0}
\lt< u_n^j - u_n^{j-1} , v_n^j \rt> + k_n \lt< A_M(v_n^{j-1},v_n^j), v_n^j \rt> = 0.
\ee
Since $v_n^j \in \pa \P_M(u_n^j)$ (which is equivalent to $u_n^j \in \pa \F_M(v_n^j)$), and since $\P_M$ is convex, there holds 
\be
 \P_M(u_n^{j-1})
 \ge  \P_M(u^j_n) +  \lt< u_n^{j-1}-u_n^j , v_n^j \rt>
 \stackrel{\eqref{eq:est.PMn.0}}= \P_M(u^j_n) + k_n \lt< A_M(v_n^{j-1},v_n^j), v_n^j \rt>.
 \label{eq:est.PMn.05}
 \ee
Moreover, $\F_M\le \F$ hence $\P_M\ge \P$, and $\P_M(u^\text{in})=\P(u^\text{in})$ when $M$ is large enough, \textit{i.e.} under condition~\eqref{M}.
Therefore, summing the estimates, we deduce: 
\[
\P(u_n^{J}) + \sum_{j=1}^J k_n \lt< A_M(v_n^{j-1},v_n^j ), v_n^j \rt> \le  \P(u^\text{in}), \qquad J \in \{1,\dots, n\}.
\]
Due to the boundary conditions, $\lt< A_M(v,v), v\rt>$ might be negative. To circumvent this difficulty and as in Section~\ref{sec.apriori}, we introduce the total free energy taking into account the energy of the charge carriers that left the domain over time. For $j \in\{ 1, \dots, n\}$, define 
\[\label{psitotm}
\P_M^{\text{\rm tot},j}= \P_M(u_n^j) + \mathcal{J}^j_n
\]
with 
\[
\mathcal{J}^j_n = \sum_{\ell=1}^j k_n \sum_{i = 1,2}\sum_{\G\in \{0,1\}}\lt[r_i^\G(T_M v_{n,i}^{\ell-1})g_i^{\G,\mu}(\xi_{n,i}^\ell-\xi_i^{\G})\xi_i^{\G}\rt](\G), \qquad 1 \le j \le n, 
\]
then one checks, using the definition of $\P_M^{\text{\rm tot}}$,~\eqref{eq:est.PMn.05} and the definition of $A_M$ that
\begin{align*}
\P_M^{\text{\rm tot},j} - \P_M^{\text{\rm tot},j-1} & = \P_M(u_n^j) - \P_M(u_n^{j-1}) + k_n \sum_{i = 1,2}\sum_{\G\in \{0,1\}}\lt[r_i^\G(T_M v_{n,i}^{j-1})g_i^{\G,\mu}(\xi_{n,i}^j-\xi_i^{\G})\xi_i^{\G}\rt](\G) \\
&\stackrel{\eqref{eq:est.PMn.05}}\le -k_n \lt< A_M(v_n^{j-1},v_n^j), v_n^j \rt> 
 \\
&\qquad +  k_n \sum_{i = 1,2}\sum_{\G\in \{0,1\}}\lt[r_i^\G(T_M v_{n,i}^{j-1})g_i^{\G,\mu}(\xi_{n,i}^j-\xi_i^{\G})\xi_i^{\G}\rt](\G)\\
&\stackrel{\eqref{eq:AM}}\le - k_n \sum_{i=1,2} \int_0^1 \sigma_i(T_Mv_{n,i}^{j-1}) \lt| \pa_x \xi_{n,i}^j \rt|^2  \,dx \\
&\qquad -  k_n \sum_{i=1,2} \sum_{\G \in \{0,1\}} 
\lt[r_i^\G(T_M v_{n,i}^{j-1})g_i^{\G,\mu}(\xi_{n,i}^j-\xi_i^{\G})\lt( \xi_{n,i}^j - \xi_i^{\G}\rt)\rt](\G).
\end{align*}
The two terms in the right-hand side being non-positive, this leads to the following estimates, which are uniform w.r.t. $n$: 
\begin{align}\label{eq:est.PMn.1}
\max_{1 \le j \le n} \P^{\text{\rm tot},j} \le&\; \P(u^\text{in}), 
\\
\label{eq:est.PMn.2}
0\, \le\, \sum_{\ell = 1}^j k_n \sum_{i=1,2} \int_0^1 \sigma_i(T_M v_{n,i}^{\ell-1}) \lt| \pa_x \xi_{n,i}^{\ell} \rt|^2  dx \le &\; \P(u^\text{in}),
\\
\label{eq:est.PMn.3}
0\, \le\,  \sum_{\ell = 1}^j k_n \sum_{i=1,2} \sum_{\G \in \{0,1\}} 
\lt[r_i^\G(T_M v_{n,i}^{\ell-1})g_i^{\G,\mu}(\xi_{n,i}^\ell-\xi_i^{\G})\lt( \xi_{n,i}^\ell - \xi_i^{\G}\rt)\rt](\G)
 \le &\;\P(u^\text{in}).
\end{align}
Since $\P^{\text{\rm tot},j}$ can be negative due to the boundary flux contributions, 
some further work on~\eqref{eq:est.PMn.1} is needed to get some bound on $\P(u_{n,i}^j)$.
Testing~\eqref{pmn} with $(0,\hat \xi_1, \hat \xi_2)$ where $\hat \xi_i$ is defined for $i=1,2$ by $x\in[0,1] \mapsto \hat \xi_i(x) = \xi_i^0 + x (\xi_i^1 - \xi_i^0)$ (and then summing the first $j$ time steps), we obtain
\be\label{eq:Jjn}
-\mathcal{J}^j_n =   \sum_{i=1,2} \int_0^1 (u_{n,i}^j - u_{i}^\text{in}) \hat \xi_i dx
+
 \sum_{\ell = 1}^j k_n  \sum_{i=1,2} (\xi_i^1 - \xi_i^0) \int_0^1 \sigma_i(T_M v_{n,i}^{\ell-1}) \pa_x \xi_{n,i}^\ell dx.
\ee
On the one hand, due to the Young-Fenchel inequality $ab \le \p_i(a) + \f_i(b)$, there holds 
\[
  \int_0^1 (u_{n,i}^j - u_{i}^\text{in}) \hat \xi_i dx \le   {\| u_i^\text{in}\|}_1 \max_\G\, |\xi_i^\G| + \dfrac12 \int_0^1 \lt[
 \p_i(u_{n,i}^j) + \f_i(2 \hat \xi_i) \rt] dx.
\]
Since the functions $\f_i$ are non-negative, we deduce
\be\label{eq:Jjn.1}
 \sum_{i=1,2} \int_0^1 (u_{n,i}^j - u_{i}^\text{in}) \hat \xi_i dx \le  \dfrac12 \P(u_n^j) + c.
\ee 
On the other hand, the elementary Young inequality $ab \le  (\e/2) a^2 + (1/2\e) b^2$ yields 
\begin{multline}
\label{eq:Jjn.32}
 \sum_{\ell = 1}^j k_n  \sum_{i=1,2} (\xi_i^1 - \xi_i^0) \int_0^1 \sigma_i(T_M v_{n,i}^{\ell-1}) \pa_x \xi_{n,i}^\ell dx 
\\
 \le   \sum_{\ell = 1}^j k_n \sum_{i=1,2} \sigma_i(T_M v_{n,i}^{\ell-1}) \lt| \pa_x \xi_{n,i}^\ell \rt|^2dx 
 + \sum_{\ell = 1}^j k_n   \sum_{i=1,2} \dfrac{ (\xi_i^1 - \xi_i^0)^2}4  \int_0^1 \sigma_i(T_M v_{n,i}^{\ell-1}) dx.
\end{multline}
One directly infers from the boundedness of $\sigma_1$ that 
\be\label{eq:Jjn.33}
\sum_{\ell = 1}^j k_n  \dfrac{ (\xi_1^1 - \xi_1^0)^2}2 \int_0^1 \sigma_1(v_{n,1}^{\ell-1}) dx \le c\, T.
\ee
Since $\sigma_2(T_M v_{n,2}^{\ell - 1}) = d_2 u_{n,2}^{\ell-1} \le \p_2(u_{n,2}^{\ell-1}) + \f_2(d_2)$, we get
\be\label{eq:Jjn.35}
\sum_{\ell = 1}^j k_n  \dfrac{ (\xi_2^1 - \xi_2^0)^2}2 \int_0^1 \sigma_2(v_{n,2}^{\ell-1}) dx \le c\lt(1+\sum_{\ell = 1}^j k_n \P(u_n^{\ell-1})\rt). 
\ee
Collecting~\eqref{eq:est.PMn.1} and~\eqref{eq:Jjn}--\eqref{eq:Jjn.35}, we deduce, for $1\le j\le n$,
\[
\P(u_n^j) =\P^{\text{\rm tot}, j}-\mathcal{J}^j_n 
\le \P(u^{\text{\rm in}}) + c (1 +T) + \dfrac12 \P(u_n^j) + \sum_{\ell=1}^jk_n\P(u_n^{l-1}).
\]
Applying a discrete Gronwall lemma after having combined the above calculations, we obtain
\be\label{eq:PMn.4}
\P(u_n^j)  \le c_T.
\ee
We deduce from~\eqref{eq:PMn.4} the following estimates on $\lt(u_n, v_n\rt)$ which are uniform with respect to $n$.
\bk
\be\label{eq:bound.v0}
 \|v_{n,0}\|_{L^{\oo}(I;H^1(0,1))} \le \cter{cte:v0}, 
\ee
\be\label{eq:bound.u1}
0 < \ov u_1 e_1(-M) \le u_{n,1} \le \ov u_1 e_1(M) < \ov u_1, 
\ee
\be\label{eq:bound.u2}
0 < \ov u_2 e_2(-M) \le u_{n,2} \le \ov u_2 e_2(M) < +\oo.
\ee

With these bounds on $u_{n1},u_{n,2}$ and the definitions of the functions $r_i^\Gamma$ and $g_i^\Gamma$, we deduce from~\eqref{eq:est.PMn.2} and~\eqref{eq:est.PMn.3} that
\[
\|\xi_{n,i}\|_{L^2(I; H^1(0,1))} \le c_M.
\]
Note that the four above estimates are uniform w.r.t. $n$, and that the quantity $\cter{cte:v0}$ appearing in~\eqref{eq:bound.v0} is the one appearing in the statement of Proposition~\ref{expm}.
Besides, recalling $\xi_{n,i}=v_{n,i}+z_i v_{n,0}$ we get for $i=1,2$,
\[
\|v_{n,i}\|_{L^2(I;H^1(0,1))}\le c_{M,T}.
\]
By continuity of $A_M: V\times V \to V^*$, we end up with
\[
\|\Delta_n u_n\|_{L^2(I;V^*)}=\|A_M(v_n,v_n)\|_{L^2(I;V^*)}\le c_{M,T},
\]
and $ U_n \in C(I;V^*)$ with the bound $\|U_n\|_{C(I;V^*)}\le c_{M,T}$.
\end{proof}

{\textbf{Step 3.}} {\rd The third step of the proof of Proposition~\ref{expm} consists in showing the following lemma.
\begin{lemma}\label{lem:step3}
There exists a solution $(u,v)$ to~\eqref{eq:PM} such that, up to a subsequence, 
\begin{align*}
u_{n,i} \underset{n\to +\oo}\longrightarrow u_i &\quad \text{almost everywhere and in the $L^\oo((0,1) \times I)$-weak-$\star$ sense, for $i=0,1,2$, } \\
v_{n,i} \underset{n\to+\oo} {-\hspace{-4pt}-\hspace{-4pt}\rightharpoonup}  v_i & \quad \text{weakly in $L^2(I;H^1(0,1))$, $i  = 1,2$}, \\
v_{n,0}  \underset{n\to +\oo}\longrightarrow v_0  &\quad \text{strongly in $L^2(I;H^1(0,1))$}.
\end{align*}
Moreover, there exists $c_T>0$ not depending on $M$ such that 
\be\label{eq:Pu}
\lt\| \P(u) \rt\|_{L^\oo(I)} \le c_T.
\ee
\end{lemma}
}
\begin{proof}
Let $(u_n,v_n)$ be the unique solution to~\eqref{eq:PMn} given by Lemma~\ref{expmn}. By Lemma~\ref{buvn} we have that, up to a subsequence, the following convergences hold as $n$ goes to $+\oo$:
\be\label{limvnzn}
\lt\{\begin{array}{lll}
\displaystyle
\rd v_{n,i} \ru v_i & \text{weakly in } L^2(I;H^1(0,1)), \; i = 1,2, \\[2mm]
\displaystyle 
\rd v_{n,0} \ru v_0 & \rd \text{in the } L^\oo(I;H^1(0,1))\text{-weak-}\star\; \text{sense},\\[2mm]
\displaystyle
U_n \ru u & \text{weakly in }  L^2(I;H^*)\text{ and } H^1(I;V^*),\\[2mm]
\displaystyle
U_n(t) \ru u(t) & \text{weakly in } {\rd V^*}
, \text{ for every } t\in I,\\[2mm]
\displaystyle
\Delta_n u_n \ru \dot u & \text{weakly in }  L^2(I;V^*). 
\end{array}\rt.
\ee
Note that, since $U_n(0)=u^\text{in}$, we have $u(0)=u^\text{in}$. Moreover, by definition of $U_n$ we have
\be\label{znun}
\|U_n-u_n\|_{L^2(I;V^*)}\le k_n\|\Delta_n u_n\|_{L^2(I;V^*)} \rw 0,\  \text{ as }n\rw +\oo.
\ee
Hence, for almost every $t$ in $I$,
\[
(U_n-u_n)(t) \underset{n\to+\oo} \longrightarrow 0  \text{ in } V^*\\[2mm]
\qquad \text{and} \qquad 
u_n(t) \ \underset{n\to+\oo} {-\hspace{-4pt}-\hspace{-4pt}\ru} u(t)\;  \text{weakly in } {\rd V^*.}
\]
By construction we have that
\be\label{uvni}
u_{n,i}={\ov u}_ie_i(T_M v_{n,i}), \ i=1,2.
\ee 
Then since $e_i \circ T_M$ are Lipschitz continuous, the bounds on $v_{n,i}$ ensure that $u_{n,i}$ are bounded in $L^2(I;H^1(0,1))$.
We deduce then from the nonlinear Aubin-Simon compactness result \cite[Proposition $1$]{Moussa_2016} that
\be\label{uniae}
u_{n,i}\rw u_i\ \text{ almost everywhere in }I\times(0,1), \text{ with } u_i={\ov u}_ie_i(T_M v_i), \ i=1,2,
\ee
which, thanks to~\eqref{eq:charge.n}, gives also $u_{n,0}\rw u_0$ almost everywhere in $I\times(0,1)$ with 
\be\label{eq:charge}
u_0 = \sum_{i=1,2} z_i u_i + \rho_\text{\rm hl}.
\ee
Together with~\eqref{eq:bound.u1} and \eqref{eq:bound.u2},  this implies that
\be\label{limun}
u_{n}\rw u\ \text{ strongly in }L^2(I;L^2(0,1))\text{ and }L^2(I;V^*),
\ee
due to the Lebesgue dominated convergence theorem,  {\rd as well as in the $L^\oo((0,1) \times (0,T))$-weak-$\star$ sense}.
Moreover, in view of~\eqref{znun}, one has
\[U_{n,i}\underset{n\to+\oo} \longrightarrow u_i \text{ strongly in }L^2(I;(H^1(0,1))^*),\ i=1,2.\]
\rd 
Since $u_{n,0} = {(E_M v_n)}_0$, one has 
\be\label{eq:vn0}
\lambda^2\int_0^1 \pa_x v_{n,0} \pa_x \tilde{v}dx + \sum_{\G\in \{0,1\}}\lt[\lt(\beta^\G v_{n,0}-f^\G \rt)\tilde{v}\rt](\G)=\int_0^1 u_{n,0} \tilde{v}dx, \qquad \forall\, \tilde{v} \in V.
\ee
The aforementioned convergence properties are sufficient to pass to the limit $n\to+\oo$ in the above equality, leading to 
\be\label{eq:v0}
\lambda^2\int_0^1 \pa_x v_{0} \pa_x \tilde{v}dx + \sum_{\G\in \{0,1\}}\lt[\lt(\beta^\G v_{0}-f^\G \rt)\tilde{v}\rt](\G)=\int_0^1 u_{0} \tilde{v}dx, \qquad \forall\, \tilde{v} \in V,
\ee
or equivalently $u_0 = {(E_Mv)}_0$.
Choosing $\tilde v= v_{n,0}$ as a test function in~\eqref{eq:vn0} and passing to the limit $n\to +\oo$ provides 
\begin{align*}
\lambda^2\int_0^T\int_0^1 \lt| \pa_x v_{n,0} \rt|^2 dxdt &+\; \int_0^T \sum_{\G\in \{0,1\}} \beta^\G  \lt| v_{n,0} \rt|^2 dt  \\ &= \int_0^T \int_0^1 u_{n,0} v_{n,0}dxdt + \int_0^T \sum_{\G\in \{0,1\}} f^\G v_{n,0} dt
\\
&\underset{n\to +\oo} \longrightarrow \int_0^T \int_0^1 u_{0} v_{0}dx dt + \int_0^T\sum_{\G\in \{0,1\}} f^\G v_{0} dt
\\
&= \lambda^2\int_0^T\int_0^1 \lt| \pa_x v_{0} \rt|^2 dx dt + \int_0^T \sum_{\G\in \{0,1\}} \beta^\G  \lt| v_{0} \rt|^2 dt
\end{align*}
thanks to~\eqref{eq:v0}. As a consequence, $\|v_{n,0}\|_{L^2(I;H^1(0,1))}$ tends to $\|v_{0}\|_{L^2(I;H^1(0,1))}$, whence 
\be\label{limv0}
v_{n,0} \underset{n\to +\oo} \longrightarrow v_0 \quad\text{strongly in} \; L^2(I;H^1(0,1)).
\ee

Next, due to~\eqref{uniae} and to~\eqref{limv0} combined with~\eqref{eq:bound.u1} and \eqref{eq:bound.u2}, 
one gets that 
\[
\P(u_n) \underset{n\to+\oo} \longrightarrow \P(u) \quad \text{in}\; L^1(I).
\]
Due to~\eqref{eq:PMn.4}, we infer that~\eqref{eq:Pu} holds true.
\bk

Finally we want to show that $(u,v)$ is a solution to problem~\eqref{eq:PM}.
First we prove that
\be\label{limam}
\lim_{n\rw +\oo} A_M(\tau_n v_n,v_n)=A_M(v,v) \text { in }L^2(I; V^*).
\ee
Taking a test function $\tilde{v} \in V$ and setting, for $i=1,2$, $\xi_{n,i}=v_{n,i}+z_i v_{n,0}$, $\tilde{\xi}_i=\tilde{v}_i +z_i \tilde{v}_0$, we have
\begin{multline*}
\displaystyle
\lt< A_M(\tau_n v_n,v_n),\tilde{v}\rt>=\sum_{i=1,2}\int_0^1 \sigma_i(T_M \tau_n v_{n,i}) \pa_x \xi_{n,i} \pa_x \tilde{\xi}_i dx\\
\displaystyle
+\sum_{i=1,2}\sum_{\G\in \{0,1\}}\lt[ r_i^\G(T_M \tau_n v_{n,i})g_i^{\G,\mu}(\xi_{n,i}-\xi_i^{\G})\tilde{\xi}_i\rt](\G).
\end{multline*}
By~\eqref{uvni},~\eqref{limun} and the properties of the translation operator function $\tau_n$ {\rd (see for instance \cite[Lemma 4.3]{Brezis11})} 
we have 
$$T_M\tau_n v_{n,i} \rw T_M v_i \hbox{ strongly in } L^2(S;L^2(0,1)) \hbox{ and in } L^2_\text{\rd loc}(S;H^s(\G)),\  i=1,2, s>\frac{1}{2},$$
implying that 
\begin{equation}\label{limsvn}
\sigma_i(T_M \tau_n v_{n,i})\rw \sigma_i(T_M v_{i})\hbox{ strongly in } L^2(S;L^2(0,1)),\  i=1,2,
\end{equation}
since $\sigma_i \circ T_M$ is Lipschitz continuous, 
and
\begin{equation}\label{limrvn}
r_i^\G(T_M \tau_n v_{n,i})\rw r_i^\G(T_M v_{i})\hbox{ strongly in } L^2(S),\  i=1,2,\ \G\in\{0,1\}
\end{equation}
thanks to the trace theorem. 
Moreover, by~\eqref{limvnzn} we obtain
\begin{equation}\label{limxni}
\xi_{n,i} \ru \xi_i  \hbox{ weakly in } L^2(S;V)
\end{equation}
with $\xi_i=v_i+z_i v_0$.
This, together with~\eqref{limsvn}, gives
\begin{equation}\label{lims}
\lim_{n\rw +\oo}\sum_{i=1,2}\int_0^1 \sigma_i(T_M \tau_n v_{n,i}) \pa_x \xi_{n,i} \pa_x \tilde{\xi}_i dx=\sum_{i=1,2}\int_0^1 \sigma_i(T_M v_{i}) \pa_x \xi_{i} \pa_x \tilde{\xi}_i dx,
\end{equation}
We are now interested in the weak convergence of the term $g_i^{\G,\mu}(\xi_{n,i}-\xi_i^{\G})$. 
The approximation $g_i^{\G,\mu}$ of $g_i^\G$ has been tailored in \eqref{eq:giGmu} so that the function $g_i^{\G,\mu, NL}$ defined by 
\[
g_i^{\G,\mu, NL}(\xi) = g_i^{\G,\mu}(\xi) - \left(g_i^{\G,\mu}\right)'(\mu) \, \xi 
\]
is constant outside $[-\mu,\mu]$. 
In view of the uniform estimate~\eqref{eq:bound.v0} of $v_{n,0}$ and the one-dimensional Sobolev-Morrey inequality (recall the definition~\eqref{eq:H1} of the $H^1(0,1)$-norm),
\[
\|w\|_{\oo} \leq \ \|w\|_{H^1(0,1)}, \qquad \forall w \in H^1(0,1),
\]
there holds 
\[
{\|v_{n,0}\|}_{\oo} \leq \cter{cte:v0} , \qquad \forall n \geq 1. 
\]
Setting  $\mu_i=M-z_i \cter{cte:v0} -|\xi_i^\G|$ and $\mu\leq\min_{i=1,2} \mu_i$,
then $|\xi_{n,i}-\xi_i^{\G}|\leq \mu$ implies $|v_{n,i}|\leq M$. Hence $g_i^{\G,\mu}(\xi_{n,i}-\xi_i^{\G})$ can be written as the sum of a linear and a nonlinear functions as follows:
\be\label{eq:giGmu.decomp}
g_i^{\G,\mu}(\xi_{n,i}-\xi_i^{\G})= \left(g_i^{\G,\mu}\right)'(\mu)\, (\xi_{n,i}-\xi_i^{\G})+g_i^{\G,\mu,NL}(T_M v_{n,i}+z_i v_{n,0}-\xi_i^{\G}).
\ee
By~\eqref{limxni} we immediately get that 
\begin{equation}\label{gil}
\left(g_i^{\G,\mu}\right)'(\mu)\, (\xi_{n,i}-\xi_i^{\G}) \underset{n\to+\oo}{-\hspace{-4pt}-\hspace{-5pt}\ru} \left(g_i^{\G,\mu}\right)'(\mu)\, (\xi_{i}-\xi_i^{\G}) \hbox{ weakly in }L^2(S).
\end{equation}
As for the nonlinear one, from~\eqref{uniae} and~\eqref{limv0} we have that $T_M v_{n,i}+z_i v_{n,0}$ converges almost everywhere in $S\times (0,1)$ to $T_M v_{i}+z_i v_{0}$, for $i=1,2$. Whence, via the Lebesgue convergence dominated theorem we also obtain that 
\begin{equation}\label{ginl}
g_i^{\G,\mu,NL}(T_M v_{n,i}+z_i v_{n,0}-\xi_i^{\G})\underset{n\to+\oo}\longrightarrow g_i^{\G,\mu,NL}(T_M v_{i}+z_i v_{0}-\xi_i^{\G}) \hbox{ strongly in }L^2(S).
\end{equation}
Combining~\eqref{gil} and \eqref{ginl} in \eqref{eq:giGmu.decomp} shows that 
\be\label{eq:giGmu.conv}
g_i^{\G,\mu}(\xi_{n,i}-\xi_i^{\G})) \underset{n\to+\oo}{-\hspace{-4pt}-\hspace{-5pt}\ru} g_i^{\G,\mu}(\xi_{i}-\xi_i^{\G}))  \hbox{ weakly in }L^2(S).
\ee
Therefore, convergences~\eqref{limrvn} and~\eqref{eq:giGmu.conv} lead to
\begin{multline}\label{limr}
\displaystyle
\lim_{n\rw +\oo}\sum_{i=1,2}\sum_{\G\in \{0,1\}}\lt[r_i^\G(T_M \tau_n v_{n,i})g_i^{\G,\mu}(\xi_{n,i}-\xi_i^{\G})\tilde{\xi}_i\rt](\G)\\
\displaystyle
\qquad \qquad =\sum_{i=1,2}\sum_{\G\in \{0,1\}}\lt[r_i^\G(T_M \tau_n v_{i})g_i^{\G,\mu}(\xi_{i}-\xi_i^{\G})\tilde{\xi}_i\rt](\G).
\end{multline}
Then, by collecting~\eqref{lims} and~\eqref{limr}, we obtain the validity of~\eqref{limam}.

Moreover, since $(u_n,v_n)$ is a solution to $(P_{M_n})$ we deduce from~\eqref{limvnzn} and~\eqref{limam} that
\[
\dot u+A_M(v,v)= \lim_{n \rw +\oo} \Delta_n u_n + A_M(\tau_n v_n,v)=0.
\]
Eventually, due to~\eqref{uniae} and~\eqref{eq:v0}, there holds $u=E_M v$, which concludes the proof of the lemma.
\end{proof}
\rd
{\textbf{Step 4.}} To establish Proposition~\ref{expm}, it only remains to prove the following.
\begin{lemma}\label{lem:step4}
Let $(u,v)$ be as in Lemma~\ref{lem:step3}. 
There exists $c_1 >0$ and $\cter{cte:v0.Lip}>0$ depending on the final time $T$ and on the data of the continuous problem, but not on $M$, such that 
\begin{align}
&{\| v_0\|}_{L^\oo(I;H^1(0,1))}\leq \cter{cte:v0}\label{eq:v0bound.2}\\
\mbox{ and }&
{\| \pa_x v_0 \|}_{L^\oo((0,1) \times I)} \le \cter{cte:v0.Lip}\label{eq:v0.Lip2}.
\end{align}
\end{lemma}
\begin{proof}
The bound \eqref{eq:v0bound.2} is a direct consequence of the convergence results stated in the proof of Lemma \ref{lem:step3} and of the estimate \eqref{eq:bound.v0}.
We deduce from~\eqref{eq:charge} and~\eqref{eq:v0} that 
\[
- \lambda^2 \pa_{xx} v_0 = \sum_{i=1,2} z_i u_i + \rho_\text{\rm hl}.
\]
In view of~\eqref{eq:Pu}, the right-hand side in the above equation is bounded uniformly w.r.t. $M$ in $L^\oo(I;L^1(0,1))$.
Therefore, $\pa_x v_0$ is bounded uniformly w.r.t. $M$ in $L^\oo(I;W^{1,1}(0,1))$, which is continuously embedded in $L^\oo((0,1) \times I)$, hence~\eqref{eq:v0.Lip2}.
\end{proof}
The proof of Proposition~\ref{expm} is now complete.
\bk 

\section{Lower and upper bounds for the chemical potentials}\label{sec.bounds}
Let $(u,v)$ be a solution to problem~\eqref{eq:PM}. In this section we provide some ${\rd L^{\oo}_\text{loc}(\R_+\times[0,1])}$-estimates {\rd on $v_i$, $i=1,2$,} which are independent of $M$. As a consequence, by choosing the cutoff level sufficiently large, this will allow to prove that $(u,v)$ is a solution to problem~\eqref{eq:P} too.\\
The proof is based on the Moser-Alikakos iteration technique, cf.~\cite{Moser60, Alikakos79}, which consists in establishing successively $L^p$-norms, for increasing values of $p$, of some appropriate functions of the chemical potentials.

We will first show an upper bound for $v_2$, then both lower bounds for $v_1$ and $v_2$, and finally an upper bound for $v_1$ in three separated theorems. The starting point of each bootstrapping procedure will be provided by estimating an appropriate $L^2$-bound to initialize the method.

 In each part of the section, we use the  identity 
 \be
\label{id_sigma_ptx _v}
\sigma_i(T_Mv_i)\pa_xv_i\pa_xu_i=d_i(\pa_x u_i)^2
\ee
which holds almost everywhere on $(0,1)\times I$ and for $i=1,2$. Indeed, on the one hand, from the chain rule and the identitities $u_i=\bar u_i e_i(T_M v_i)$ and $\sigma_i(z)=d_i\bar u_i e_i'(z)$ we have  
\be
\label{id_sigma_ptx _v_1}
\sigma_i(T_M v_i)\pa_x [T_Mv_i]=d_i \pa_x u_i.
\ee 
On the other hand, since $\pa_xu_i=\pa_x[T_M v_i]=0$ almost everywhere on $\{|v_i|\ge M\}$ and  $\pa_x[T_M v_i]=\pa_xv_i$ almost everywhere else, there holds
\be
\label{id_sigma_ptx _v_2}
\pa_x[ T_M v_i]\pa_x u_i=\pa_xv_i \pa_xu_i.
\ee
Multiplying~\eqref{id_sigma_ptx _v_1} by $\pa_xu_i$ and using~\eqref{id_sigma_ptx _v_2} yields~\eqref{id_sigma_ptx _v}.

In the sequel, $c$ will denote different positive constants independent of $M$.

\subsection{Upper bound for $v_2$}
In order to get an upper bound for the chemical potential $v_2$, we first derive an upper bound for the density $u_2$ in the following theorem. 
\begin{teo}\label{upperu2}
Under assumptions {\rm(H$_1$)--(H$_5$)}, let $(u,v)$ a solution to problem~\eqref{eq:PM}. There exists a constant $c>0$ independent on $M$ such that
\[\|u_2(t)\|_{L^{\oo}(0,1)}\le c\quad \forall t \in I.\]
\end{teo}
\begin{proof}
Let $p\ge 2$ and $w=(u_2-k)_+$, where $k>0$ will be fixed later. We use $(0,0,p w^{p-1})$ as a test function in~\eqref{eq:PM}. We have $\dot u_2\in L^2(I,[H^1(0,1)]^*)$ and $w\in L^2(I,H^1(0,1))\cap L^\oo$, so, using a smoothing argument and classical properties of Sobolev spaces, there holds 
\[\lt<\dot u_2 ;p w^{p-1}\rt>_{(H^1)^*,H^1}
= \lt<\dot w ;pw^{p-1}\rt>_{(H^1)^*,H^1}
=\dfrac d{dt}\lt[\int_0^1w^p\, dx\rt]\qquad\text{in }L^1(I).\] 
Therefore $\lt<\dot u_2;p w^{p-1}\rt>+\lt<A_M(v,v);pw^{p-1}\rt>=0$ reads,
\be\label{equ2}
\dfrac{d}{dt}\int_0^1w^p \,dx=-Q_1-Q_2-Q_3,
\ee
with
\begin{align*}
Q_1&=\int_0^1 \sigma_2(T_M v_2) \pa_x v_2 \pa_x (pw^{p-1})\,dx,
\qquad\quad Q_2=\int_0^1 \sigma_2(T_M v_2) z_2 \pa_x v_0 \pa_x (pw^{p-1})\,dx,\\
Q_3&=\sum_{\G\in \{0,1\}}\lt[r_2^\G(T_M v_2) g_2^{\G,\mu}(\xi_2-\xi_2^{\G})pw^{p-1}\rt](\G).
\end{align*}
In order to estimate the boundary term $Q_3$, we note that $\xi_2-\xi_2^{\G}=v_2(\G)-[\xi_2^{\G}-z_2 v_0(\G)]$ and we define successively
\[v_2^{*,\G}=\xi_2^{\G}-z_2 v_0(\G),\qquad u_2^{*,\G}={\ov u}_2 e_2(T_M v_2^{*,\G})\quad\text{ and }\quad w^{*,\G}=(u^{*,\G}_2-k)_+.
\]
We can then write
\begin{align*}
Q_3&=p\sum_{\G\in \{0,1\}}\lt[r_2^\G(T_M v_2) g_2^{\G,\mu}(v_2-v_2^{*,\G}) \lt(w^{p-1}-(w^{*,\G})^{p-1}\rt)\rt](\G)\\
&\qquad +p\sum_{\G\in \{0,1\}}\lt[r_2^\G(T_M v_2) g_2^{\G,\mu}(v_2-v_2^{*,\G})(w^{*,\G})^{p-1}\rt](\G)=:Q_{31}+Q_{32}.
\end{align*}
Due to the monotonicity of the involved functions, it results $Q_{31}\ge 0$. Let us notice that, due to \eqref{eq:v0bound}, $v_2^{*,\G}$ is bounded independently of $M$. Therefore, it is possible to choose $k$ (independent of $M$) such that $w^{*,\G}=0$ for $\G\in\{0,1\}$ and we get $Q_{32}= 0$
Hence,
\be\label{t3}
-Q_3\le 0.
\ee
For $Q_1$, we use~\eqref{id_sigma_ptx _v} and the fact that $\pa_x w\pa_x u_2=(\pa_x w)^2$ to get
\be
Q_1= p(p-1)d_2 \int_0^1 (w^{(p-2)/2}\pa_x w)^2\,dx=\dfrac{4d_2 (p-1)}{p}\int_0^1 |\pa_x w^{p/2}|^2\,dx.
\ee
We treat the term $Q_2$ as follows. 
Differentiating $w^{p-1}$ and using~\eqref{id_sigma_ptx _v_1}, we get  
\[
Q_2
=p(p-1) z_2 \int_0^1 \sigma_2(T_M v_2)\pa_xw  \pa_x v_0  w^{p-2}\,dx
.
\]
From $u_2=\bar u_2e_2(T_Mv_2)$ and the definitions $e_2(z)=e^z$ and $\sigma_2(z)=d_2\bar u_2e^z$, 
\[
  \sigma_2(T_M v_2)=d_2\bar u_2 \exp(\ln(u_2/\bar u_2))=d_2 u_2,
\]
so that we get 
\[
Q_2 = d_2 p(p-1) z_2 \int_0^1  u_2w^{p-2} \pa_xv_0 \pa_x w \, dx.
\]
In any case we have $0\le u_2\le k+w$, so we have $|Q_2|\le Q_{21}+Q_{22}$ where 
\begin{align*}
Q_{21}&=d_2 p(p-1) |z_2| \|\pa_x v_0\|_\oo  \int_0^1w^{p-1}|\pa_x w| \,dx\\
Q_{22}&= k d_2 p(p-1)   |z_2|\, \|\pa_x v_0\|_\oo \int_0^1w^{p-2}|\pa_x w| \,dx.
\end{align*}
Taking into account that $\pa_x v_0$ is bounded uniformy w.r.t. $M$ (see Lemma~\ref{lem:step4}), that $p-1=(p/2-1)+ p/2$ and that $pw^{p/2-1}\pa_x w=2 \pa_x w^{p/2}$, and applying the Young inequality, we get
\begin{align}\nonumber
Q_{21}&\le d_2 \int_0^1 |\pa_x w^{p/2}|^2 \,dx+d_2 (p-1)^2 |z_2|^2\|\pa_x v_0\|^2_\oo\int_0^1 w^p \,dx\\
\label{t21}
&\le d_2 \int_0^1 |\pa_x w^{p/2}|^2 \,dx+c d_2 \,p^2 \int_0^1 w^p \,dx.
\end{align}
Arguing similarly for the term $Q_{22}$ and using additionally the Young inequality: 
\[
ab\le \dfrac{p-2}{p}a^{\frac{p}{p-2}}+\dfrac{2}{p}b^{\frac p2} \quad \text{ for }a,b\ge 0,\, p > 2,
\]
we have
\begin{multline}\label{t22}
Q_{22}\le d_2 \int_0^1 |\pa_x w^{p/2}|^2 \,dx+k^2 d_2 (p-1)^2 |z_2|^2\|\pa_x v_0\|^2_{L^{\oo}(0,1)} \int_0^1 w^{p-2} \,dx\\
\le d_2 \int_0^1 |\pa_x w^{p/2}|^2 \,dx+ d_2 c\,p^2\int_0^1 w^p \,dx+ d_2 c\,p.
\end{multline}
By collecting~\eqref{equ2}--\eqref{t22}, we obtain
\be\label{triangle}
\dfrac{d}{dt}\int_0^1w^p \,dx+ d_2 \dfrac{2p-4}{p}\int_0^1 |\pa_x w^{p/2}|^2 \,dx\le 
d_2 \cter{cte:3} p^2\int_0^1 w^p \,dx+d_2 c p,\quad  \forall p> 2.
\ee
where $\ctel{cte:3}$ and the generic constant $c$ depends neither on $M$ nor on $p$. Without loss of generality, we assume that $\cter{cte:3} \geq 1/16$.
Now we combine the following Gagliardo-Nirenberg interpolation inequality:
\[
\|\chi\|^3_{L^2(0,1)}\le c \|\chi\|^2_{L^1(0,1)}\left(\|\chi\|^2_{L^2(0,1)}+\|\pa_x\chi\|^2_{L^2(0,1)}\right)^{1/2},\qquad \forall \chi \in H^1(0,1),
\]
together with the Young inequality to get that, for $\e \leq \frac12$, 
\[
\frac12 \int_0^1 \chi^2 \,dx\le (1-\e) \int_0^1 \chi^2 \,dx\le \dfrac{c}{\sqrt{\e}}\lt(\int_0^1 |\chi|\,dx\rt)^2+\e \int_0^1 |\pa_x \chi|^2 \,dx.
\] 
We apply it with $\chi=w^{p/2}$, $\e=(p-2)/(\cter{cte:3} p^3) \in (0,1/2)$ since $p\geq 4$ and $\cter{cte:3} \geq 1/16$. Then, by the choice of $\e$, from~\eqref{triangle} we deduce that (still for $p\geq 4$)
\[\dfrac{d}{dt}\int_0^1w^p \,dx\le \dfrac{c\,p^2}{\sqrt{\e}}\lt(\int_0^1 w^{p/2} \,dx\rt)^2+c_2 p\le c\,p^3 \lt[\lt(\int_0^1 w^{p/2} \,dx\rt)^2+1\rt].\]
We integrate the last inequality with respect to the time variable and choose $k$ such that $w(0)=0$ (\textit{i.e.} $k\geq \sup u_2^{in}$) getting
\[
\int_0^1w^p (t) \,dx\le c\,p^3 \lt[\sup_{s\in I}\lt(\int_0^1 w^{p/2}(s) \,dx\rt)^2+1\rt] \quad \forall t \in I,
\]
that is,
\be\label{stella}
\|w(t)\|^p_{L^p(0,1)}\le c\,p^3 \lt[\sup_{s\in I}\|w(s)\|^p_{L^{p/2}(0,1)}+1\rt] \quad \forall t \in I\text{ and }\forall p>2.
\ee
Taking inspiration in \cite{Glitzky_2011}, we set for $m\in \mathbb{N}$:
\[
b_m=\max\left(1,\sup_{s\in I}\|w(s)\|^{2^m}_{L^{2^m}(0,1)}\right).
\]
Choosing $p=2^m$ in~\eqref{stella}, we obtain 
$$
b_m\le c\,2^{3m} b^2_{m-1}
$$
and by induction, we get
\[
 b_m \le c^{(2^{m-1}-1)}\,(2^3)^{(2^m-m-1)}\,b_1^{2^{m-1}}\quad \forall m\geq 1.
\]
Taking the power $1/2^m$ leads to the estimate:
\[
\|w(t)\|^2_{L^{2^m}(0,1)}\le c^{1-1/2^{m-1}}(2^3)^{2-(m+1)/2^{m-1}} \max\left(1,\sup_{s\in I}\|w(s)\|^2_{L^2(0,1)}\right) 
\quad \forall t \in I.
\]
Sending $m$ to $+\oo$ we find
\be\label{bornew}
\|w(t)\|_{L^{\oo}(0,1)}\le c\lt[\sup_{s\in I}\|w(s)\|_{L^2(0,1)}+1\rt] \quad \forall t \in I.
\ee
The last step consists in obtaining a $L^2$-estimate of the norm of $w(s)$. To this aim, we can consider inequality~\eqref{triangle} in the limit $p \searrow 2$:
\[
\dfrac{d}{dt}\int_0^1w^2 \,dx\le c\left( 1+  \int_0^1 w^2 \,dx \right).
\]
We apply the Gronwall's lemma to have 
\[
\int_0^1 w^2(s) \,dx\le \lt(\int_0^1 w^2(0)\,dx +1\rt)e^{cs}\le c,
\]
due to the bounds on the initial data stated in (H$_5$), independently of $M$.
Eventually, choosing 
\[
k\ge \max\lt\{\|u_2^{\text{\rm in}}\|_{L^{\oo}(0,1)},|u_2^{*,0}|,|u_2^{*,1}|\rt\},
\]
there exists a positive constant $c$ which \emph{does not depend on $M$} such that $0\le u_2\le c$ over $I$.
\end{proof}
\begin{oss}\label{upperv2}
In particular, since $u_2=\bar u_2 e_2(T_M v_2)=\bar u_2\exp(T_M v_2)$, Theorem~\ref{upperu2} implies that
\[T_M v_2 \le e_2^{-1}\lt(\dfrac c{\bar u_2}\rt)=\log\lt(\dfrac c{\bar u_2}\rt)\quad\text{almost everywhere on } I\times(0,1), \]
hence so does $v_2$ by choosing $M \geq \log\lt(\dfrac c{\bar u_2}\rt)$.

\end{oss}
\subsection{Lower bounds for $v_1$ and $v_2$}
With the following theorem we derive some lower bounds for the chemical potentials $v_1$ and $v_2$. For this, we write the chemical potentials and the mobilities appearing in the problem as functions of the corresponding carrier densities. More precisely, the chemical potentials $v_i(u_i)$, $i=1,2$, are given by~\eqref{vi}, while the $\sigma_i(u_i)$ are given by~(H$_3$) so that~\eqref{id_sigma_ptx _v} holds.
\begin{teo}\label{loweruv12}
Under assumptions {\rm(H$_1$)--(H$_5$)}, let $(u,v)$ a solution to problem~\eqref{eq:PM}. There exists a constant $c>0$ independent on $M$ such that
\[v_i\ge -c\quad\text{almost everywhere in } I\times(0,1) \text{ for } i=1,2.\]
\end{teo}
\begin{proof}
Let $p\ge 2$ and $w_i=(-T_M v_i-k)_+$ for $i=1,2$, where $k>0$ will be fixed later. Observe that $\n w_i=-\n[T_M v_i]$ almost everywhere on $\{w_i>0\}$ and $\n w_i=0$ almost everywhere in the rest of the domain.\\
 We do both calculations simultaneously and to lighten the exposition we make the following abuse of notation: we drop all subscripts $i$ except when a precise notation is required (we write $w$ for $w_i$, $u$ for $u_i$, $\sigma(T_Mv)$ for $\sigma_i(T_Mv_i)$, $\xi$ for $\xi_i$, $\xi^\G$ for $\xi_i^\G$, \textit{etc}). \\
We have using the chain rule and $u=\bar u e_i(T_Mv)$,
\[
\dfrac d{dt}\int_0^1w^p \,dx= p\int_0^1  w^{p-1} \pa_t w\ \,dx= -p\int_0^1  w^{p-1} \pa_t [T_M v]\ \,dx= -p\int_0^1 \dfrac{w^{p-1}}{\bar u e_i'(T_Mv))}\pa_t u.
\]
Since by definition, $\sigma_i(y)=d_i \bar u_i e_i'(z)$, this reads
\[
\dfrac d{dt}\int_0^1w^p \,dx=-pd\int_0^1 \dfrac{w^{p-1}}{\sigma(T_M v)}\pa_t u\, dx
\]
Applying~$\dot u + A_M(v,v)=0$ with the test function
\[
\lt(0,-\dfrac{pd_1 w_1^{p-1}}{\sigma_1(T_Mv_1)},0\rt)\quad\text{if }i=1,\qquad
\lt(0,0,-\dfrac{p d_2 w_2^{p-1}}{\sigma_2(T_Mv_2)}\rt)\quad \text{if }i=2,
\] 
we have 
\be
\label{a}
\dfrac d{dt}\int_0^1w^p \,dx= R_1 +R_2 +R_3,
\ee
where, 
\begin{align*}
R_1&=pd\int_0^1 \sigma(T_Mv) \pa_x v \pa_x \lt(\dfrac{w^{p-1}}{\sigma(T_Mv)}\rt)\,dx,
\qquad\quad R_2=pdz\int_0^1 \sigma(T_Mv)\pa_x v_0 \pa_x \lt(\dfrac{w^{p-1}}{\sigma(T_Mv)}\rt) \,dx,\\
R_3 &= pd\sum_{\G\in \{0,1\}}\lt[r^\G(T_M v) g^{\G,\mu}(\xi-\xi^{\G})\dfrac{w^{p-1}}{\sigma(T_Mv)}\rt](\G).
\end{align*}
As in the previous proof, we estimate the boundary term $R_3$ using the fact that $\xi-\xi^{\G}=v(\G)-[\xi^{\G}-z v_0(\G)]$. By defining $v^{*,\G}=\xi^{\G}-z v_0(\G)$ and $w^{*,\G}=(-v^{*,\G}-k)_+$, we write
\begin{align*}
R_3&=pd\sum_{\G\in \{0,1\}}\lt[\dfrac{r^\G(T_Mv)}{\sigma(T_Mv)} g^{\G,\mu}(v-v^{*,\G})\lt(w^{p-1}-(w^{*,\G})^{p-1}\rt)\rt](\G)\\
&\qquad \quad +pd\sum_{\G\in \{0,1\}}\lt[\dfrac{r^\G(T_Mv)}{\sigma(T_Mv)} g^{\G,\mu}(v-v^*)(w^{*,\G})^{p-1}\rt](\G)\\
&=:R_{31}+R_{32}.
\end{align*}
Since $r^\G$ and $\sigma$ are positive and $g^{\G,\mu}$ is increasing, we have $R_{31}\le 0$. By choosing $k$ large enough so that $w^{*,\G}=0$ for $\G\in\{0,1\}$, we get $R_{32}= 0$. Whence it results 
\be\label{a3}
R_3\le 0.
\ee

To treat $R_1$, we write 
\be\label{deriv_wp-1/sigma}
\pa_x\lt(\dfrac{pw^{p-1}}{\sigma(T_Mv)}\rt)=\dfrac{p(p-1)}{\sigma(T_Mv)}w^{p-2}\pa_xw + \dfrac{p\sigma'(T_M v)}{[\sigma(T_Mv)]^2} w^{p-1}\pa_xw.
\ee
The term $R_1$ then splits as follows:
\[
R_1=-d p(p-1)\int_0^1 w^{p-2}|\pa_x w|^2 \,dx - dp\int_0^1 \dfrac{\sigma'(T_Mv)}{\sigma(T_Mv)}w^{p-1}|\pa_x w|^2 \,dx=:R_{11}+R_{12}.
\]
We easily see that
\[
R_{11}=\dfrac{-4d(p-1)}p\int_0^1 |\pa_x w^{p/2}|^2\,dx.
\]
Next, for $i=2$, $\sigma_2'/\sigma_2\equiv 1$ so that $R_{12}\le0$ in this case. For $i=1$, $\sigma_1'(y)/\sigma_1(y)=-\tanh(y/2)$ so $\sigma_1'(T_Mv_1)/\sigma_1(T_Mv_1)\ge0$ on $\{w_1>0\}=\{T_Mv_1<-k\}$ and $R_{12}\le0$ also in this case. We conclude that 
\be
\label{a11} 
R_1\le\dfrac{-4d(p-1)}p\int_0^1 |\pa_x w^{p/2}|^2\,dx.
\ee

Using again~\eqref{deriv_wp-1/sigma} and $|\sigma'/\sigma|\le 1$, we have for the remaining term $|R_2|\le R_{21}+R_{22}$ with
\begin{align*}
R_{21}&=dp^2(p-1)|z|\|\pa_xv_0\|_\oo \int_0^1  w^{p-2}|\pa_x w| \,dx,\\
R_{22}&= dp^2|z|\|\pa_xv_0\|_\oo \int_0^1  w^{p-1} |\pa_xw| \,dx.
\end{align*}
Writing
\[
R_{21}=2dp(p-1)|z|\|\pa_xv_0\|_\oo \int_0^1  w^{p/2-1}|\pa_x w^{p/2}| \,dx.
\]
and using the Cauchy--Schwarz and Young inequalities, we get
\begin{align}\nonumber
|R_{21}|&\le \dfrac{d(p-1)}p\int_0^1 |\pa_x w^{p/2}|^2 \,dx+ d(p-1)p^3|z|^2 \|\pa_x v_0\|^2 \int_0^1 w^{p-2} \,dx\\
\label{a21}
&\le\dfrac{d(p-1)}p \int_0^1 |\pa_x w^{p/2}|^2 \,dx+d\, c\,p^4\lt(1 + \int_0^1 w^p \,dx\rt).
\end{align}
Similarly, 
\be\label{a22}
|R_{22}|\le\dfrac{d(p-1)}p \int_0^1 |\pa_x w^{p/2}|^2 \,dx+d\, c\,p^4\int_0^1 w^p \,dx.
\ee
By collecting~\eqref{a3}, \eqref{a11}, \eqref{a21} and \eqref{a22}, we obtain
\be\label{Moser}
\dfrac{d}{dt}\int_0^1w^p \,dx+ \dfrac{2d(p-1)}p\int_0^1 |\pa_x w^{p/2}|^2 \,dx \le d\, c p^4 \lt(1+\int_0^1 w^p \,dx \rt), \quad \forall p\ge 2.
\ee
As in the proof of Theorem~\ref{upperu2}, we deduce from this estimate that
\[
\|w(t)\|_{L^{\oo}(0,1)}\le c\lt[\sup_{s\in I}\|w(s)\|_{L^2(0,1)}+1\rt] \quad \forall t \in I,
\]
for some constant $c>0$ independent of $M$. To initialize the process, we consider~\eqref{Moser} with $p=2$ and apply Gronwall's lemma to deduce the bound 
\[
\sup\{\|w(s)\|_{L^2(0,1)}:s\in I\}\le c, 
\]
for some constant $c$ depending on the $T$, $c_1$, $c_2$ and the $L^\oo$-norm of $v^{\text{\rm in}}$ (bounded thanks to (H$_5$)). 

As a conclusion, choosing 
\[
k\ge \max\{\|v_1^{in}\|_{L^{\oo}(0,1)}, \|v_2^{in}\|_{L^{\oo}(0,1)},|v_1^{*,0}|, |v_1^{*,1}|,|v_2^{*,0}|, |v_2^{*,1}|\},
\] 
there holds $v_i(t)\ge -c$ almost everywhere in $I\times(0,1)$ and for  $i=1,2$ with a constant $c>0$ independent of $M$.
\end{proof}

\subsection{Upper bound for $v_1$}
We state here an upper bound for the chemical potential $v_1$.
\begin{teo}\label{upperv1}
Under assumptions~{\rm(H$_1$)--(H$_5$)}, let $(u,v)$ a solution to problem~\eqref{eq:PM}. There exists a constant $c>0$ independent on $M$ such that
\[v_1(t)\le c,\ \forall t \in I.\]
\end{teo}
\begin{proof}
Let $p\ge 2$ and $w=(T_M v_1-k)_+$, with 
\[
k\ge\max \lt\{ \|v_1^{\text{\rm in}}\|_{L^{\oo}(0,1)},|v_1^{*,0}|,|v_1^{*,1}|\rt\}
\] 
where we have set $v_1^{*,\G}=\xi_1^{\G}-z_1 v_0(\G)$ for $\G\in\{1,2\}$.\\
 We compute 
\be\label{proof_upbound_v1}
\dfrac d{dt}\int_0^1w^p \,dx= p\int_0^1w^{p-1}\pa_t[T_Mv_1]\ \,dx=pd\int_0^1 \dfrac{w^{p-1}}{\sigma_1(T_M v_1)}\pa_t u_1\, dx.
\ee
Here we can apply~$\dot u + A_M(v,v)=0$ with the test function
\[
\lt(0,\dfrac{p w^{p-1}}{\sigma_1(T_Mv_1)},0\rt),
\]
and express the right hand side of~\eqref{proof_upbound_v1} as a sum of integrals involving space derivatives of $w$. The proof follows exactly the same lines of that of Theorem~\ref{loweruv12}, we do not repeat the details. %
\end{proof}

{
\section{Conclusion}\label{sec:conclu}

\subsection{Proof of Theorem~\ref{teoex}}\label{sec:concluproof}

In this paper, we have introduced a new corrosion model inspired from the DPCM introduced in \cite{Eacta}. The changes we introduced are motivated by the expected compatibility of the model with thermodynamics. Indeed, they permit to establish the decay of a free energy with a control of the dissipation of energy as stated in Section~\ref{sec.apriori}.

The main result of the paper is the existence of a weak solution to this new corrosion model, stated in Theorem~\ref{teoex}. In Section~\ref{sec.PM}, we have first established the existence of a solution to a regularized problem $(P_M)$. Then lower and upper bounds for the chemical potentials $(v_i)_{i=1,2}$ have been proved in Section~\ref{sec.bounds}, so that we are now able to conclude the proof of Theorem~\ref{teoex}.

Indeed, to find a solution to problem~\eqref{eq:P} it suffices to show the existence of a solution on any finite time interval of the form $I=[0,T]$, $T>0$. Let us fix such $T>0$ and let $M>0$. Proposition~\ref{expm} ensures the existence of a solution $(u,v)$ to the regularized problem~\eqref{eq:PM} on $I$. Now, Theorems~\ref{upperu2},~\ref{loweruv12},~\ref{upperv1} and Remark~\ref{upperv2} guarantee the existence of bounds for $\|v_i\|_{L^{\oo}(I\times [0,1])}$, $i=1,2$ independent of $M$. Consequently, for $M$ large enough, for this solution the operators $A_M$ and $E_M$ coincide with $A$ and $E$ respectively, and $(u,v)$ turns out to be a solution to the original problem~\eqref{eq:P}.

\subsection{Towards a comparison between the DPCM and the new model}\label{sec:DPCMvsvDPCM}

Some numerical methods have been introduced for the simulation of the DPCM in \cite{JCP} and implemented in the code CALIPSO. For the simplified two-species DPCM on a fixed domain, the numerical scheme is described and analyzed in \cite{CHLV-Colin}. It is based on a Scharfetter-Gummel approximation of the linear drift-diffusion fluxes.
The main change we made when designing the new model (we will call in the sequel vDPCM) relies on the new definition of the fluxes of cations which are now nonlinear with respect to the densities of cations. For the approximation of this new flux, we consider an extension of the SQRA (SQuare-Root Approximation) introduced in \cite{SQRA}. The scheme will be introduced and analyzed in further work. 

In order to compare the two models experimentally, we consider the test case proposed in \cite[Appendix A]{BCHZ}, restricted to 2-species. We keep the same values for all the parameters, even if the changes we made should induce new values of parameters for the vDPCM. The $pH$ (needed for the computation of some parameters) is set equal to 8.5.  

Figure~\ref{Fig:Courant_Tension} shows the evolution of the total current (which is a combination of the electronic and the cationic currents) with respect to the applied potential $V$ (expressed in Volts and evaluated relatively to the electrode reference NHE). We observe that the solutions to the DPCM and the vDPCM have similar qualitative behaviours albeit quantitatives differences.

\begin{figure}[ht]
\centering
\includegraphics[width=0.6\textwidth]{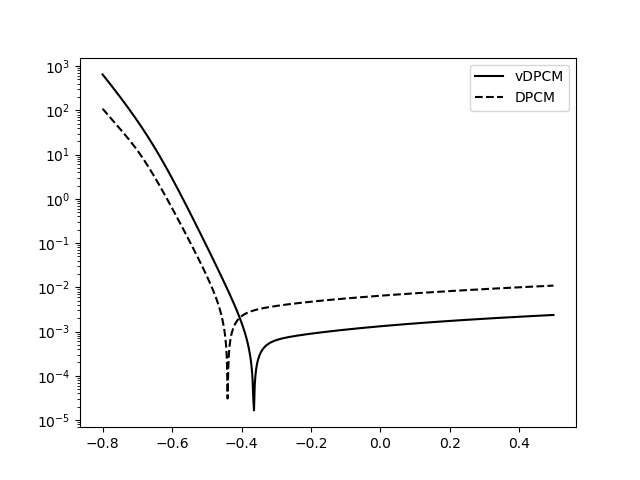}
\caption{Evolution of the total current for the steady state (in physical units $A\cdot m^{-2}$) in terms of the applied potential (in Volts) at $pH=8.5.$}
\label{Fig:Courant_Tension}
\end{figure}

Figure~\ref{Fig:profiles} shows the profiles of the densities of cations and electrons and of the electrostatic potential obtained with the two models at two different times, still for a $pH$ equal to $8.5$. The applied potential is equal to $0.3$ Volts. We chose a first time $t=18$s at which the system is still evolving and a second time $t=1510 s$ at which a steady state is reached for both models. We observe once again the quantitative differences between the two models, while showing that these differences stay controlled.
\begin{figure}[ht]
\begin{tabular}{ccc}
\includegraphics[width=0.3\textwidth]{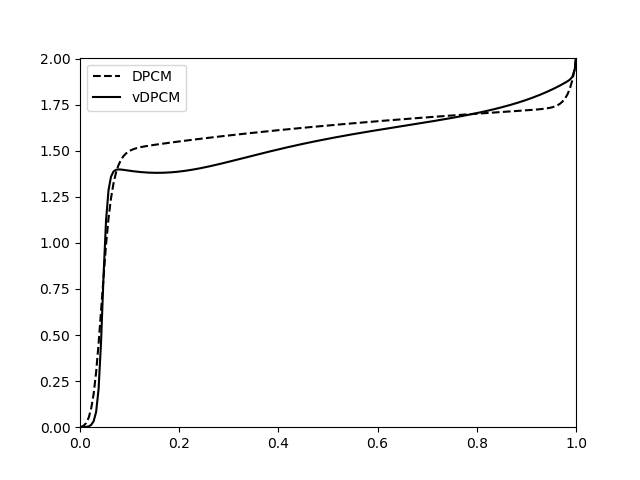}&
\includegraphics[width=0.3\textwidth]{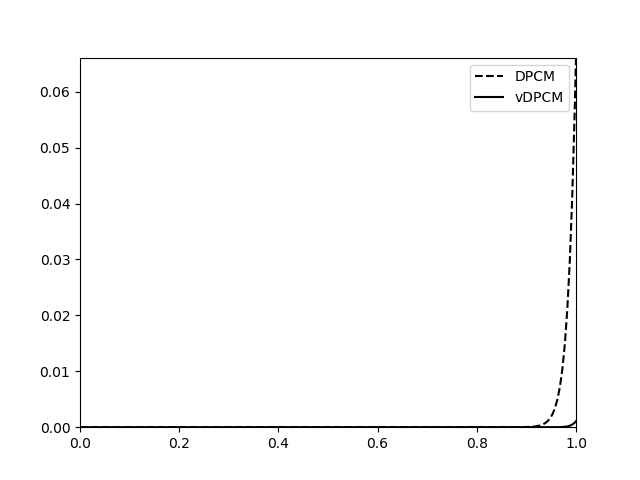}&
\includegraphics[width=0.3\textwidth]{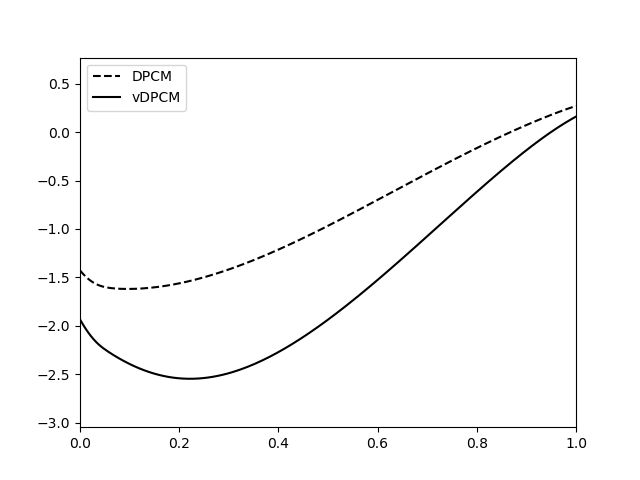}\\
density of cations&density of electrons&electrostatic potential\\
\includegraphics[width=0.3\textwidth]{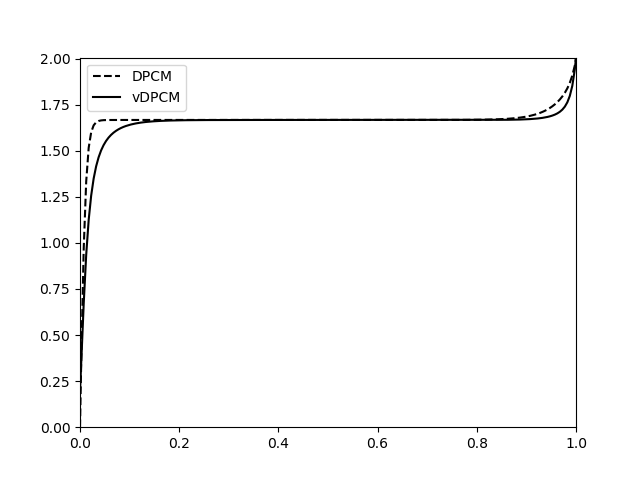}&
\includegraphics[width=0.3\textwidth]{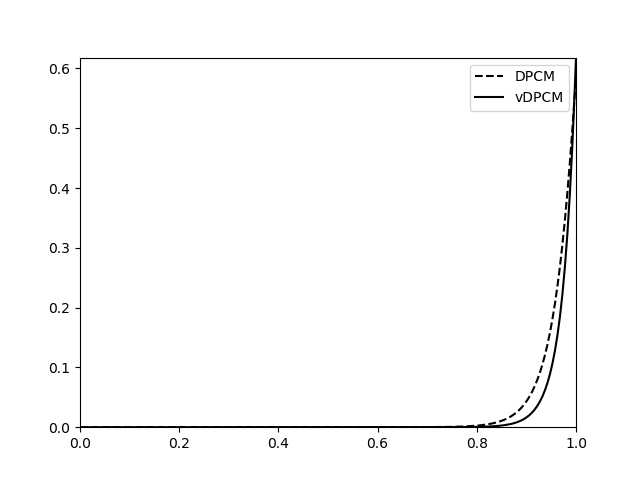}&
\includegraphics[width=0.3\textwidth]{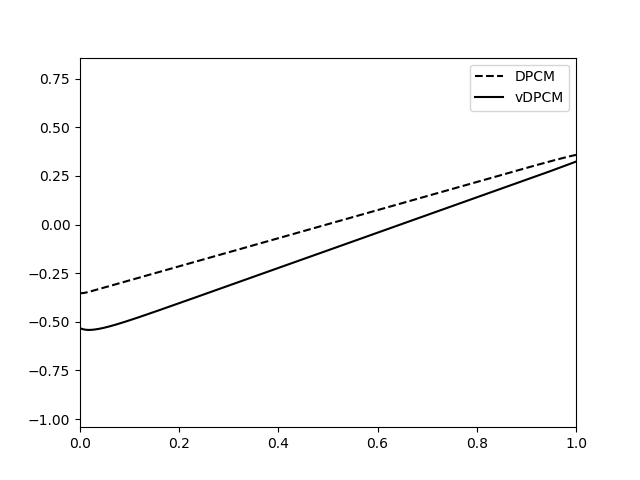}\\
density of cations&density of electrons&electrostatic potential
\end{tabular}
\caption{Profiles of the scaled densities of cations, electrons and of the electrostatic potential (in Volts, the physical unit) at two different times: $t=18$s at the top, $t=1510$s at the bottom (steady-state).}
\label{Fig:profiles}
\end{figure}

As we have established in this paper the well-posedness of the vDCPM without any restriction on the set of parameters, we will be able to introduce the numerical methods and make more numerical investigations in further work. Moreover, one big challenge for the future consists  in adapting the present work to the 3-species DPCM \cite{Eacta} on moving domain. 
}

\subsection*{Aknowledgements} The authors warmly thank Christian Bataillon for his kind feedback on the model. 
This project has received funding from the European Union's Horizon 2020 research and innovation programme under grant agreement
 No 847593 (WP DONUT), and was further supported by Labex CEMPI (ANR-11-LABX-0007-01). C. Cancès also 
 acknowledges support from the COMODO project (ANR-19-CE46-0002) and C. Chainais-Hillairet from the MOHYCON project (ANR-17-CE40-0027-01). 
 J. Venel warmly thanks the Inria research center of the University of Lille for its hospitality.

\end{document}